\documentclass[11pt]{amsart}
\usepackage{amsmath,amsthm,amssymb,  graphicx, tikz, float, verbatim} 
\usepackage{anysize}
\usepackage{xcolor, longtable}

\usetikzlibrary{arrows,shapes,positioning}

\newcommand\sizeone{0.4cm}
\newcommand\sizetwo{1.8cm}
\newcommand\sizethree{2.1cm}
\newcommand\sizefour{2.8cm}
\newcommand\sizefive{3.5cm}

\newcommand\dd[2]{\centering $\delta^{#1}(1-\delta)^{#2}$}

\marginsize{1.5cm}{1.5cm}{1cm}{2.5cm}

\DeclareMathOperator{\Aut}{Aut}
\DeclareMathOperator{\diam}{diam}
\DeclareMathOperator{\fix}{fix}
\DeclareMathOperator{\Sym}{Sym}
\DeclareMathOperator{\Alt}{Alt}
\newcommand{\expec}{\mathbb{E}}
\newcommand{\Prob}{\mathbb{P}}
\newcommand{\Z}{\mathbb{Z}}
\newcommand{\supp}{{\rm supp}}
\newcommand{\prob}{{\rm Prob}}
\newtheorem{thm}{Theorem}[section]

\newtheorem{conj}[thm]{Conjecture}

\newtheorem{cor}[thm]{Corollary}
\newtheorem{lemma}[thm]{Lemma}
\newtheorem{defn}[thm]{Definition}
\numberwithin{equation}{thm}

\newcommand{\fixed}{\mathsf{fixed}}

\newcommand{\ourmid}{\;\Bigl\vert\;}

\begin{document}

\title{Bounds on the diameter of Cayley graphs of the Symmetric group}

\author[J.~Bamberg]{John Bamberg}
\author[N.~Gill]{Nick Gill}
\author[T.~P.~Hayes]{Thomas P.~Hayes}
\author[H.~A.~Helfgott]{Harald A.~Helfgott}
\author[\'A.~Seress]{\'Akos Seress}
\author[P. Spiga]{Pablo Spiga}

\address{John Bamberg\newline
 School of Mathematics and Statistics\newline
 University of Western Australia \newline 
35 Stirling Highway, Crawley, WA 6009, Australia}
\email{john.bamberg@uwa.edu.au}

\address{Nick Gill\newline
Department of Mathematics\newline
The Open University\newline
Walton Hall, Milton Keynes, MK7 6AA, United Kingdom}
\email{n.gill@open.ac.uk}

\address{Thomas P. Hayes\newline
Department of Computer Science\newline
Mail stop: MSC01 1130\newline
1 University of New Mexico\newline
Albuquerque, NM 87131-0001}
\email{hayes@cs.unm.edu}

\address{Harald A.~Helfgott\newline
D\'epartement de math\'ematiques et applications\newline
\'Ecole normale sup\'erieure\newline
45 rue d'Ulm \newline F-75230 Paris,France}
\email{helfgott@dma.ens.fr}

\address{\'Akos Seress\newline
 School of Mathematics and Statistics\newline
 University of Western Australia \newline 
35 Stirling Highway, Crawley, WA 6009, Australia
\newline
and\newline
Department of Mathematics\newline
The Ohio State University \newline
231 W. 18th Avenue, Columbus, OH 43210, USA}
\email{akos@math.ohio-state.edu}

\address{Pablo Spiga\newline
Dipartimento di Matematica e Applicazioni\newline
University of Milano-Bicocca\newline
Via Cozzi 53, 20125 Milano, Italy} \email{pablo.spiga@unimib.it}


\begin{abstract}
In this paper we are concerned with the conjecture that, for any set of generators $S$ of the symmetric group $\Sym(n)$, the word length in terms of $S$ of every permutation is bounded above by a polynomial of $n$.
We prove this conjecture for sets of generators containing a permutation fixing at least $37\%$ of the points.
\end{abstract}

\maketitle

\section{Introduction}
For a group $G$ and a set $S$ of generators of $G$, we write
$\Gamma(G,S)$ for the \emph{Cayley graph} of $G$ with connection set $S$,
that is, the graph with vertex set $G$ and with edge set $\{\{g,sg\}\mid
g\in G,s\in S\}$. The \emph{diameter} $\diam(\Gamma)$ of a graph
$\Gamma$ is the maximum
distance among the vertices of $\Gamma$ and, in the case of a Cayley graph
$\Gamma(G,S)$, it is the maximum (over the group elements $g\in G$) of the
shortest expression $g=s_1^{i_1}\cdots s_m^{i_m}$, with $s_k\in S$ and
$i_k\in \{-1,1\}$. We define the \emph{diameter of a group $G$} as 
\[
\diam(G):= \max \{ \ \diam (\Gamma(G,S)) \mid S \mbox{ generates } G \ \}.
\]

A first investigation of the diameter of Cayley graphs for general groups was undertaken by Erd\H{o}s and R\'enyi~\cite{ErdosRenyi}. Later Babai and Seress~\cite{babaiseress} obtained asymptotic estimates on $\diam(G)$ depending heavily on the group structure of $G$. In particular the results in~\cite{babaiseress} highlight the discrepancy between the diameter of Cayley graphs of groups close to being abelian and the diameter of Cayley graphs of non-abelian simple groups.  Moreover, \cite{babaiseress} contains the following conjecture of Babai.

\begin{conj}[{\cite[Conjecture~$1.7$]{babaiseress}}]\label{babaiconj}
There exists $c>0$ such that, for all non-abelian simple groups $G$,
$\diam(G) \leq (\log |G|)^c$.
\end{conj}

The conjecture remains open, although significant progress has been made. In particular, starting with the work of Helfgott on the groups $\textrm{PSL}(2,p)$ and $\textrm{PSL}(3,p)$~\cite{helfgott2,helfgott3} and based thereon, there has been a series of results~\cite{bgt,helfgill, ps} proving the conjecture for finite simple groups of Lie type of bounded rank. The best statement known at the time of writing is by Pyber and Szab\'o~\cite{ps} and says that there exists a polynomial $c$ such that, for a finite simple group $G$ of Lie type
of Lie rank $r$, we have
$\diam(G)\leq (\log|G|)^{c(r)}$. For the sake of comparison,
Conjecture~\ref{babaiconj} asserts that $c$ should be a constant rather than a polynomial. 

The proofs of these theorems make use of new results in additive
combinatorics, specifically on growth in simple groups.  
We note that the difficulties in generalizing these results from
groups of bounded rank to those of unbounded rank seem closely related
to difficulties in proving Conjecture~\ref{babaiconj} for the alternating
groups $\Alt(n)$. In both cases (that is, classical groups of
unbounded rank and alternating groups) there are known counterexamples
to  general ``growth results" for sets (see for example~\cite{ppss,ps,sp}), which were central to the approach used to prove Conjecture~\ref{babaiconj} for groups of Lie type of bounded rank. What is more, these two classes of counterexample are, in some sense, related.

In this paper we focus on the case where $G=\Alt(n)$ or $\Sym(n)$. 
Let $\Omega$ be a set of size $n$. For $g\in \Sym(\Omega)$, define the \emph{support} of $g$ by
$\supp(g)=\{\gamma\in\Omega  \mid  \gamma^g\neq \gamma\}$. Observe that $\supp(g)$ is equal to the complement in $\Omega$ of the {\it fixed set}, $\fix(g)$, of $g$.
Babai, Beals and Seress~\cite{bbs} proved the following result.

\begin{thm}[{\cite{bbs}}]\label{t: bbs}
For every $\varepsilon<1/3$ there exists $c_\varepsilon>0$ such that, if $G=\Sym(n)$ or $\Alt(n)$ and $S$ is a set of generators of $G$ containing an element $g$ with $|\supp(g)| \le \varepsilon n$, then
$$\diam(\Gamma(G,S)) \leq c_\varepsilon n^8.$$
\end{thm}

In this paper we provide a variant of the argument in~\cite{bbs} to
prove the following stronger theorem.  

\begin{thm}\label{main1}
Let $C=0.63$. 
There exists $c>0$ such that, if $G=\Sym(n)$ or $\Alt(n)$ and $S$ is a set of generators of $G$ containing an element $g$ with $|\supp(g)| \le Cn$, then 
$$\diam(\Gamma(G,S)) \le O(n^{c}).$$
\end{thm}

We do not try to minimize the exponent $c$ in the theorem. Our arguments give $c\le 78$, but with some more work the bound on $\diam(\Gamma(G,S)) $ can be improved to at least $O(n^{66})$.

Theorem~\ref{main1} also extends to directed graphs. Given $G=\langle S \rangle$, the {\em directed Cayley graph}
$\vec{\Gamma}(G,S)$ is the graph with vertex set $G$ and edge set $\{ (g,sg)
: g \in G, s \in S \}$. Analogously to the undirected case, the diameter of $\vec{\Gamma}(G,S)$ is defined as the maximum (taken over $g \in G$) of the shortest expression $g=s_1\cdots s_m$, with each $s_k \in S$. By a theorem of Babai~\cite[Corollary~2.3]{babai},
$\diam(\vec{\Gamma}(G,S)) = O\left(\diam(\Gamma(G,S) ) \cdot (\log |G|)^2 \right)$ 
holds for all groups $G$ and sets $S$ of generators, so we immediately obtain the following corollary.

\begin{cor}
\label{directed version}
Let $C=0.63$. 
There exists $d>0$ such that, if $G=\Sym(n)$ or $\Alt(n)$ and $S$ is a set of generators of $G$ containing an element $g$ with $|\supp(g)| \le Cn$, then
$$\diam(\vec{\Gamma}(G,S)) \leq O(n^{d}).$$
\end{cor}

We note that for arbitrary sets of generators the best known bound is quasipolynomial, by a recent result of Helfgott and Seress:

\begin{thm}[{\cite{aph}}]\label{hsmain}
For $G=\Alt(n)$ and $\Sym(n)$, $\diam(G) =\exp(O( (\log n)^4 \log\log n))$.
\end{thm}

The machinery developed in this paper turns out to have application to other questions within permutation group theory. Indeed it is possible to use variants of the results given in Section~\ref{sec: machine} to recover, and strengthen, classical results concerning multiply transitive groups due to Manning \cite{manning1, manning2, manning3} and Wielandt \cite{wielandt2}. This will be the subject of a forthcoming paper \cite{ghss}.

\subsection{The main ideas}
It is well-known and easy to see that if a set $A$ of generators of $G=\Alt(n)$ or $\Sym(n)$ contains a $3$-cycle $t$ then every element of $G$ can be written as a word of length less than $n^4$ in $A$. Indeed, repeatedly conjugating $t$ by $A$ gives all 3-cycles as words of length less than $n^3$, and each element of $\Alt(n)$ is a product of at most $\lfloor n/2 \rfloor$ $3$-cycles. Finally, if $G=\Sym(n)$ then one more multiplication by $A$ gives words for all elements of $G$. Thus, given any set $S$ of generators of $G$, in order to prove Conjecture~\ref{babaiconj}, it is enough to construct a 3-cycle as a word in $S$ of polynomial length. 

We may try to reach a 3-cycle in stages, by constructing elements of smaller and smaller support. Up to very recently, the only subexponential method to obtain an element of support less than $cn$, for some constant $c<1$, from arbitrary generating sets was in \cite{babaiseress88}. In that paper, iteration of the support reduction was utilized to prove $\diam(G) =\exp( (1+o(1))\sqrt{n \log n} )$, the only subexponential bound until Theorem~\ref{hsmain}.

Theorem~\ref{t: bbs} may be interpreted as a reduction of Conjecture~\ref{babaiconj} for alternating groups, to the problem of constructing an element $g$ of moderately small support as a short word in an arbitrary set $S$ of generators. The proof of Theorem~\ref{t: bbs} in \cite{bbs} is based on the following observations.
\begin{itemize}
\item[(BBS1)] If $|\supp(a)|<\varepsilon n$ for some $\varepsilon <1/3$, $a \in G$, $G=\Alt(n)$ or $\Sym(n)$, and $r$ is a random element of $G$ then, for $b=a^r$, the commutator $[a,b]=a^{-1}b^{-1}ab$ has support smaller than $a$ with positive probability.
\item[(BBS2)] In (BBS1), it is not necessary that $r$ is a uniformly distributed random element of $G$. It is enough that, for some constant $\ell$, $r$ maps a sequence of distinct elements of length $\ell$ from the permutation domain nearly uniformly to all other sequences. Furthermore, random words $r$ on any set $S$ of generators of $G$, of length $n^{O(\ell)}$, satisfy this property.
\end{itemize}
In \cite{bbs}, the number $\ell=3$ was chosen and a 3-cycle was constructed in $O(\log\log n)$ applications of (BBS1). A major conceptual novelty of \cite{bbs} is that besides the natural action of $G$ on $n$ points and the action of $G$ on itself as in $\Gamma(G,S)$, it is beneficial to work with other actions of $G$. This principle is more clearly formulated in \cite{aph}. In \cite{bbs} and in the present paper, the action of $G$ on sequences of length $\ell$ from the natural permutation domain is used, while in \cite{aph} other actions are utilized as well. 

Chronologically, we made three improvements to the argument in \cite{bbs}.
\begin{itemize}
\item[(NEW1)] The conclusion of (BBS1) holds for $\varepsilon < 1/2$, implying a version of Theorem~\ref{main1} with $C<0.5$.
\item[(NEW2)] With positive probability, the commutator $[a,b]$ has many fixed points, and also contains a significant number of $3$-cycles. Thus, if $|\supp(a)|<\varepsilon n$ for some $\varepsilon < 0.585$, then $[a,b]^3$ has support smaller than $a$. This implies Theorem~\ref{main1} with $C=0.585$.
\item[(NEW3)] With positive probability, the permutation $[a,b^{-1}][a,b]$ has many fixed points and $2$-, $3$-, $4$- and $5$-cycles. So, for $\varepsilon \le 0.63$, $([a,b^{-1}][a,b])^{60}$ has support smaller than $a$. 
\end{itemize}
In this paper we only prove the strongest version based on (NEW3). 
We have to overcome several technical difficulties: (i) the analysis of the local behaviour (i.e., finding how $a$ and $b$ should interact on small subsets $\Delta$ of the natural domain such that $[a,b^{-1}][a,b]$ forms a short cycle on some points of $\Delta$); (ii) ensuring that $([a,b^{-1}][a,b])^{60}$ is not the identity of $G$; and (iii) handling the special case when the originally given generator $a$ has order $2^x3^y$ for some $x,y \ge 0$. We shall apply the argument of (BBS2) with $\ell=26$.

The structure of this paper is as follows. In Section~\ref{basic concepts}, we collect basic concepts regarding groups and graphs, and introduce the central notion of {\em $\alpha\beta$-trees}. These are the objects describing the possible local interactions of $a$ and $b$. We also introduce our probabilistic method. In Section~\ref{sec: machine} we give a graph-theoretic technique for estimating the number of fixed points of $w(a,b)$, where $w(\alpha,\beta)$ is a reduced word in $\alpha$ and $\beta$, and $a$ and $b$ are particular conjugate permutations of $\Sym(\Omega)$. In Section~\ref{sec3}, we apply the results of Section~\ref{sec: machine} to the word $[\alpha,\beta^{-1}][\alpha,\beta]=\alpha^{-1}\beta\alpha\beta^{-1}\alpha^{-1}\beta^{-1}\alpha\beta$ and we prove Theorem~\ref{main1}. In Section~\ref{sec: closing} we discuss some possible extensions of Theorem~\ref{main1}.

\subsection{Acknowledgements}

All authors would like to thank Gordon Royle for useful discussions in the initial phase of this research. In addition expenses for N.G. and H.H. to visit the University of Western Australia were paid for by a UWA Research Collaboration Award which was awarded to a team including Prof. Royle. We are, therefore, doubly grateful to Prof. Royle for without this financial support it is unlikely that this research would have been undertaken.

A.S is supported in part by the NSF and by ARC Grant DP1096525. N.G. was a frequent visitor to the University of Bristol during the period of this research and is grateful for the generous support he received from the maths department there.

\section{Basic concepts}\label{basic concepts}

In this section we collect definitions and basic results that will be needed in the proof of Theorem~\ref{main1}. 

\subsection{Permutation groups}\label{s: perm group} 
Let $\Omega = \{ 1,2,\ldots,n\}$. We use $\Sym(n)$ and $\Sym(\Omega)$ interchangeably; more exactly, we use $\Sym(\Omega)$ when we emphasise the action of $\Sym(n)$ on $\Omega$. 
Let $S \subseteq  \Sym(n)$ and $k, l\in\mathbb{Z}^+$. Define
\begin{equation*}
S^\ell=\{s_1\cdots s_\ell \, \mid \, s_1, \dots, s_\ell\in S\}; \, \, \, \,  S^{-1}=\{ s^{-1} \mid s \in S \}.
\end{equation*}
For $\omega \in \Omega$ and $a,g \in \Sym(n)$, we write $\omega^g$ for the image of $\omega$ under $g$; and $a^g$ for $g^{-1}ag$.
We denote by $\Omega_{(k)}$ the set of $k$-tuples of
distinct elements of $\Omega$ and we write $n_{(k)}=|\Omega_{(k)}|=n(n-1)\cdots
(n-k+1)$.

 We shall use the following result of J.~Whiston~\cite{whiston}.

\begin{lemma}[{\cite{whiston}}]\label{reduce genset}
Any set $S$ of generators for $G=\Sym(n)$ or $\Alt(n)$ contains a subset $A$ of cardinality less than or equal to $n-1$ that also generates $G$.
\end{lemma}

If $A$ and $B$ are two sets of generators for $G$ with $A \subseteq B$ then clearly $\diam(\Gamma(G,B)) \le \diam(\Gamma(G,A))$. Therefore, Lemma~\ref{reduce genset} implies that it is enough to prove Theorem~\ref{main1} for sets $S$ of generators of size at most $n$ which contain an element $g$ of small support.

\subsection{Graphs} \label{s: graphs}
In this paper a {\em graph} $X$ is a finite connected {\em directed} graph. Moreover, we allow loops on the vertices of $X$ and multiple edges. We do not assume that $X$ is {\em strongly connected}, i.e., it is possible that there is no directed path between some vertices $x$ and $y$ of $X$. We write $V(X)$ for the set of vertices of $X$ and $E(X)$ for the set of edges of $X$. An edge $e$ running from vertex $i$ to  vertex $j$ will be written $(i,j)$ but we warn that this notation is ambiguous as there may be more than one such edge.

We define an $\alpha\beta$-\emph{graph}, $T$ say, to be a graph together with a label, $\alpha$ or $\beta$, attached to every edge. We require that for each vertex $v \in V(T)$ and for each $\gamma \in \{ \alpha,\beta\}$, $v$ is incident with at least one edge labelled by $\gamma$. 
\footnote{We impose this condition because it is necessary in some of the arguments used in Section \ref{sec: machine}. It is not {\it a priori} necessary to our strategy for analysing the fixed points of words; this condition limits our implementation to those reduced words $w$ on $\{\alpha, \beta, \alpha^{-1}, \beta^{-1}\}$ for which all exponents are equal to $\pm 1$.}
We also require that at most one edge starting at $v$ is labelled by $\gamma$ and at most one edge ending at $v$ is labelled by $\gamma$. Here a loop at $v$ counts as one incoming and one outgoing edge for $v$. Notice that all vertices of an $\alpha\beta$-graph have in-degree at most $2$ and out-degree at most $2$. For $\gamma \in \{ \alpha,\beta \}$, we define $T_\gamma$ to be the subgraph of $T$ with vertex set $V(T)$ and edge set the set of edges of $T$ labelled by $\gamma$. 

We say that a cycle $C$ of $T$ is {\em monochromatic} if all of its edges
are labelled $\alpha$ (resp. $\beta$), that is, $C$ is a sequence of vertices $(v_1,\ldots,v_r,v_{r+1})$ with $v_{r+1}=v_1$ and $r\geq 2$, and
where for each $i\in \{1,\ldots,r\}$ the ordered pair $(v_i,v_{i+1})$
is an edge of $T$ labelled $\alpha$ (resp. $\beta$). 

Given permutations $a,b\in\Sym(\Omega)$ and an injective map $\iota:V(T)\to \Omega$, we say that $T_\alpha$ is {\em hosted by} $(\iota, a)$ if $(x\iota)^a=y\iota$ for each edge or loop $(x,y)\in E(T_\alpha)$. Similarly,  $T_\beta$ is {\em hosted by} $(\iota, b)$ if $(x\iota)^b=y\iota$ for each edge or loop $(x,y)\in E(T_\beta)$. Finally, $T$ is {\em hosted by} $(\iota,a,b)$ if $T_\alpha$ is hosted by $(\iota, a)$ and $T_\beta$ is hosted by $(\iota, b)$.

Let $T$ be an $\alpha\beta$-graph. Observe that for $\gamma \in \{ \alpha,\beta \}$, the connected components of $T_\gamma$ are of three types: 
\begin{enumerate}
 \item $\gamma$-\emph{loops}: isolated vertices, namely the vertices $v$ of $T$ having a loop at $v$ labelled with $\gamma$;
\item $\gamma$-\emph{cycles}: monochromatic cycles all of whose edges are labelled $\gamma$;
\item $\gamma$-\emph{paths}: maximal directed paths such that all edges are labelled with $\gamma$.
\end{enumerate}
We denote by $l_\gamma(T)$ the number of $\gamma$-loops and by $p_\gamma(T)$ the number of $\gamma$-paths and $\gamma$-cycles. 

We say that an $\alpha\beta$-graph is an $\alpha\beta$-{\em tree} if all undirected cycles are monochromatic (and so necessarily they are also directed cycles). Note that an $\alpha\beta$-tree may not be a tree in the usual graph-theoretic sense. The following result will be crucial.

\begin{lemma}\label{tree}
Let $T$ be an $\alpha\beta$-graph. Then
$p_\alpha(T)+p_\beta(T)+l_\alpha(T)+l_\beta(T)\leq |V(T)|+1$. Moreover, equality holds if and only if $T$ is an $\alpha\beta$-tree. 
\end{lemma}

\begin{proof}
Let $B$ be the graph with vertex set the set of $\alpha$-loops, $\alpha$-paths, $\alpha$-cycles, $\beta$-loops, $\beta$-paths and $\beta$-cycles of $T$. We declare two distinct vertices $x$ and $y$ of $B$ adjacent if there exists a vertex $v$ of $T$ such that $v$ is incident with both $x$ and $y$. By construction, $B$ is bipartite (with the $\alpha$-objects comprising one class of the bipartition and the $\beta$-objects the other) and $|V(B)|=p_\alpha(T)+p_\beta(T)+l_\alpha(T)+l_\beta(T)$. By the definition of $\alpha\beta$-graphs, each vertex $v$ of $T$ is incident with exactly one component of $T_\alpha$ and with exactly one component of $T_\beta$. Hence $v$ defines exactly one edge in $B$ and so $|E(B)|=|V(T)|$. Since $T$ is connected, the graph $B$ is connected and so $|V(B)|\leq |E(B)|+1=|V(T)|+1$, proving the first claim.

Observe that if $T$ contains a non-monochromatic cycle, then $B$ contains a cycle. Conversely if $B$ contains a cycle, then $T$ must contain a non-monochromatic cycle. We conclude that $B$ is a tree if and only if $T$ has no non-monochromatic cycles, and the second claim follows from the standard fact that $B$ is a tree if and only if $|V(B)|=|E(B)|+1$.
\end{proof}

We end this subsection with some more definitions. For an $\alpha\beta$-graph $T$ and for $0<\delta<1$, we define $\delta_T := (1-\delta)^{l_\alpha(T)+l_\beta(T)}\delta^{p_\alpha(T)+p_\beta(T)}$. An {\em isomorphism} between $\alpha\beta$-graphs $T_1$ and $T_2$ is defined to be a bijection $\varphi: V(T_1) \to V(T_2)$ that preserves edges and edge-labels. If, for some $x,y \in V(T_1)$, there are two directed edges from $x$ to $y$ in $T_1$ (with necessarily different labels) then in $T_2$ there are also two directed edges from $x\varphi$ to $y\varphi$. We define an \emph{automorphism} of $T$ to be an isomorphism between $T$ and itself. The set of all such permutations is the \emph{automorphism group} $\Aut(T)$.

\subsection{Words and graphs}\label{s: words and graphs} Let $T$ be an $\alpha\beta$-graph and $w=w_1w_2\cdots w_k$ be a reduced word with $w_1,\dots, w_k\in \{\alpha,\alpha^{-1},\beta, \beta^{-1}\}$. We say
that $T$ {\em admits} $w$, if there exists a vertex $x$ of $T$ such that by starting at
$x$ and by tracing the edges and the loops of $T$ with labels $(w_1,w_2,\ldots,w_k)$,
we visit {\em all} vertices and edges of $T$ and we return to the vertex $x$ (here
by abuse of notation, we interpret the label $\alpha^{-1}$ as the label $\alpha$ with the edge pointing in the opposite
direction, and a similar convention holds for $\beta$). The vertex $x\in V(T)$ is called a \emph{fixed vertex} for $(T,w)$; note that there may be more than one such vertex in $T$.

The following Table~\ref{table: w itself only} contains pairwise non-isomorphic $\alpha\beta$-{\it trees} that admit the word $w=[\alpha,\beta^{-1}][\alpha,\beta]$. Here dashed (red) lines are labelled with $\alpha$, solid (blue) lines are labelled with $\beta$, and for simplicity of drawing all loops are omitted; thus, any vertex which is not incident to a red line (or blue line, respectively) is in fact incident to an $\alpha$-loop (or $\beta$-loop, respectively). The fixed vertices are written as red stars, and under each graph $T$ we have written the value for $\delta_T$. 

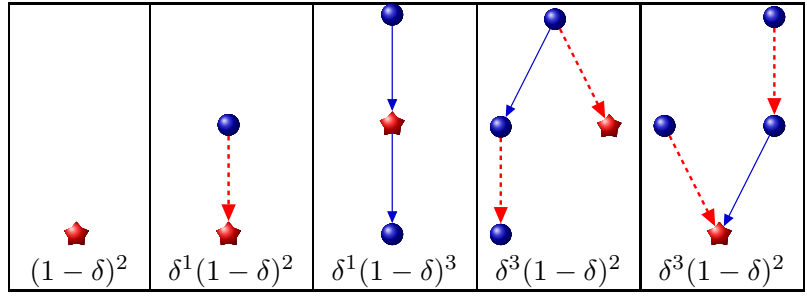
\begin{table}[H]
\caption{$\alpha\beta$-trees admitting $w=[\alpha,\beta^{-1}][\alpha,\beta] = \alpha^{-1}\beta\alpha\beta^{-1}\alpha^{-1}\beta^{-1}\alpha\beta$.}\label{table: w itself only}
 \begin{tabular}{|c|c|c|c|c|}
\hline 
 \parbox[b]{1.5cm}{\resizebox{\sizeone}{!}{\input{1cycles/tikz/g1.tex}} \centering $(1-\delta)^2$}
 &
 \parbox[b]{\sizetwo}{\resizebox{\sizeone}{!}{\input{1cycles/tikz/g2.tex}} \dd{1}{2}}
 &
 \parbox[b]{\sizetwo}{\resizebox{\sizeone}{!}{\input{1cycles/tikz/g3.tex}} \dd{1}{3}}
 &
 \parbox[b]{\sizetwo}{\resizebox{\sizetwo}{!}{\input{1cycles/tikz/g4.tex}} \dd{3}{2}}
 &
 \parbox[b]{\sizetwo}{\resizebox{\sizetwo}{!}{\input{1cycles/tikz/g5.tex}} \dd{3}{2}}
 \\
 \hline
 \end{tabular}
 \end{table}

Next, we explain how $\alpha\beta$-trees can be used to estimate the number of fixed points in certain permutations. Let $T$ be an $\alpha\beta$-tree admitting the reduced word  $w=w_1w_2 \cdots w_k$, and let $x$ be a fixed vertex of $(T,w)$. Starting at $x$ and tracing $w$ we obtain a sequence $U=(x=x_0,x_1,\ldots, x_k=x)$ of vertices of $T$ such that, by definition, all $v \in V(T)$ occur in $U$. 

Let $a,b \in \Sym(\Omega)$. If $T$ is hosted by 
$(\iota, a, b)$ for some injective map $\iota: V(T) \to \Omega$ then, starting at $x\iota$ and tracing the word $w$ by using $a$ and $b$ for the labels $\alpha$ and $\beta$, respectively, we obtain the sequence $U\iota =(x_0\iota, x_1\iota,\ldots, x_k\iota)$. In particular, $x\iota$ and $w$ uniquely determine the entire map $\iota$ and $x\iota$ is a fixed point of the permutation $w(a,b)$. 

For fixed $a,b$, and $w$ as above, let $T_1,\ldots,T_m$ be pairwise non-isomorphic $\alpha\beta$-trees admitting $w$. For $1 \le j \le m$, we denote the number of fixed vertices in $T_j$ (with respect to $w$) by $\fixed(T_j)$. Also, let $I_j$ be an index set such that for $z \in I_j$, there exists $\iota_z: V(T_j) \to \Omega$ with $T_j$ hosted by $(\iota_z,a,b)$.

\begin{lemma}
\label{number of fixed points}
Let $a,b,w,m,T_j,I_j$ be as in the previous paragraph.
\begin{itemize}
\item[$(i)$] If $x$ is a fixed vertex of $(T_{j_1},w)$, $y$ is a fixed vertex of $(T_{j_2},w)$, and $x\iota_{z_1}=y\iota_{z_2}$ for some $z_1 \in I_{j_1}$ and $z_2 \in I_{j_2}$ then $j_1=j_2$ and $x,y$ are in the same orbit of $\Aut(T_{j_1})$.
\item[$(ii)$] The number of fixed points of $w(a,b)$ is at least
\begin{equation*}
|\fix(w(a,b))| \ge \sum_{j=1}^m \frac{|I_j| \cdot \fixed(T_j)}{|\Aut(T_j)|}.
\end{equation*}
\end{itemize}
\end{lemma}

\begin{proof}
(i) Since $x\iota_{z_1}=y\iota_{z_2}$ and $w$ determines $T_{j_1}\iota_{z_1} = T_{j_2}\iota_{z_2}$, the map $\iota_{z_1}\iota_{z_2}^{-1}$ is a label-preserving isomorphism between $T_{j_1}$ and $T_{j_2}$. Therefore, $j_1=j_2$ and $\iota_{z_1}\iota_{z_2}^{-1} \in \Aut(T_{j_1})$. Moreover, since $x\iota_{z_1}\iota_{z_2}^{-1}=y$, $x$ and $y$ are in the same orbit of $\Aut(T_{j_1})$.

(ii) For $1 \le j \le m$, let $F_j$ be the set of fixed points of $w(a,b)$ of the form $x\iota_z$, for some fixed vertex $x \in V(T_j)$ and $z \in I_j$. By (i), the sets $F_j$ are pairwise disjoint. Moreover, for $z \in I_j$, 
$\iota_z$ contributes $\fixed(T_j)$ fixed points to $F_j$, giving a total count of $|I_j| \cdot \fixed(T_j)$. Part (i) also implies that
any element of $F_j$ occurs in this count at most $\Aut(T_j)$ times.
\end{proof}

\subsection{Walks on graphs}\label{s: walks on graphs} In this subsection we consider graphs $X$ that are symmetric and regular in the following sense. {\em Symmetric} means that for any two vertices $x,y \in V(X)$, the number of edges from $x$ to $y$ is the same as the number of edges from $y$ to $x$. {\em Regular} of valency $d$ means that each $x \in V(X)$ has in-degree and out-degree $d$. For $x \in V(X)$, we denote by $\Delta(x)$ the $d$-element multiset $\{ y \mid (x,y) \in E(X) \}$.

\begin{defn} 
\label{def: lazy}
{\em A \emph{lazy random walk} on $X$ is a discrete stochastic process where a particle moves from vertex to vertex in $X$. If, after $k$ steps, the particle is at $x \in V(X)$ and $\Delta(x)= \{ y_1 , \ldots , y_d \}$ then at step $k+1$ the particle
\begin{itemize}
\item[(i)] stays at $x$ with probability $1/2$;
\item[(ii)] moves to vertex $y_i$ with probability $1/(2d)$, for all $i = 1, \ldots, d$.
\end{itemize}
}
\end{defn}

The asymptotic rate of convergence for the probability
distribution of a particle in a lazy random walk on $X$ is an important and well-studied problem in combinatorics and computer science (see \cite{Lov}). For $x, y  \in V(X)$, we  write $p_k (x, y)$ for the probability that the particle is at vertex $y$ after $k$ steps of a lazy random
walk starting at $x$. For a fixed $\varepsilon > 0$, the {\em mixing time for $\varepsilon$} is the minimum value of $k$ such that
\[ 
\frac{1}{|V(X)|} (1-\varepsilon) \le p_k (x, y) \le \frac{1}{|V(X)|} (1+\varepsilon)    \] 
for all $x, y \in V(X)$. The following estimate is well-known; for a proof, see e.g. \cite[Section 4]{aph}.

\begin{lemma}
\label{l: mixing time}
Let $X$ be a connected, symmetric, and regular directed graph of valency $d$ and with $N$ vertices, and let $\varepsilon >0$. Then the mixing time for $\varepsilon$ is at most $N^2 d \log(N/\varepsilon)$.
\end{lemma}

For $G=\Sym(\Omega)$ or $\Alt(\Omega)$ and $G=\langle S \rangle$, we are interested in the following symmetric and regular directed graphs $X_k$, for positive integers $k$. Let $V(X_k):=\Omega_{(k)}$, and $E(X_k):=\{ (x,x^g) \mid x \in 
\Omega_{(k)} {\mbox{ and }}g \in S \cup S^{-1} \}$. Clearly $X_k$ has $n_{(k)}$ vertices, is connected, is symmetric, and is regular of valency $|S \cup S^{-1}|$. 

It is useful to induce random walks on the graphs $X_k$ for different $k$ at the same time. This is done as follows. First, we choose a subset $J \subseteq \{ 1,2,\ldots,\ell \}$, where $J$ is the set of steps when the particle moves to a neighbour of the current position as in Definition~\ref{def: lazy}(ii). The length $j:=|J|$ is chosen from the binomial distribution $\textrm{B}(\ell, 1/2)$ and then $J$ itself is chosen from the uniform distribution on the $j$-element subsets of $\{ 1,2,\ldots,\ell \}$. Finally, for $i \in J$, we choose $g_i \in S \cup S^{-1}$ uniformly and use $g_i$ in the $i^{\mathrm{th}}$ step to define the edge on which the particle moves. The overall effect, that is, the trajectory of a lazy random walk with initial position $x \in V(X_k)$, is the same as computing the image of $x$ under the permutation $r=\prod_{i \in J} g_i$; we say that the permutation $r$ {\em is realised by} the lazy random walk. The construction of $r$ uses only the number $\ell$ and $S \cup S^{-1}$, so the permutation $r$ can be considered as realised by lazy random walks in more than one graph $X_k$. Of course, these lazy random walks are not independent. 

Lemma~\ref{l: mixing time} will be useful to us in the following form.

\begin{lemma}\label{lemmaMarkovChain}
Let $S$ be a set of generators of $\Sym(\Omega)$ or $\Alt(\Omega)$ of cardinality at most $n$, and let $k$ be a positive integer. Fix $0<\varepsilon<1$ and set $\ell\geq 2n^{2k+1}\log(n^k/\varepsilon)$. Then, if $r \in \Sym(\Omega)$ is realised by a lazy random walk of length $\ell$ on $X_k$, and $x, y \in \Omega_{(k)}$, then 
\[
(1-\varepsilon)\frac 1{n_{(k)}} \leq \Prob(x^r=y) \leq(1+\varepsilon)\frac 1{n_{(k)}}.
\]
\end{lemma} 

\begin{proof}
Recall that $n_{(k)}=|V(X_k)|$ by definition, and $n_{(k)}\le n^k$. Furthermore $|S \cup S^{-1}| \le 2n$ and the proof follows from Lemma~\ref{l: mixing time}. 
\end{proof}

Note that if $r\in \Sym(\Omega)$ is realised by a lazy random walk of length $\ell$ on $X_k$, then $r\in (S\cup S^{-1}\cup \{1\})^\ell$.

\section{Primary machinery}\label{sec: machine}

The results of this section provide the primary machinery for a proof of Theorem~\ref{main1}. The main step of the proof is that given a set $S$ of generators for $G=\Sym(n)$ or $\Alt(n)$ and $a \in S$ of support size $|\supp(a)|=\delta n$, we would like to construct a permutation as a short word in $S$ with support size less than $\delta n$. The permutations we consider are of the form 
$w(a,a^r)$, for an appropriately chosen reduced word $w$ in the symbols $\{ \alpha, \alpha^{-1}, \beta, \beta^{-1} \}$. In this section, we assume that $w$ is given, and describe how to choose $r \in G$ such that $w(a,a^r)$ has many fixed points. We obtain $r$ as a permutation realised by a lazy random walk. 

Let $w=w_1w_2\cdots w_k$ and let ${\mathcal T}=\{ T_1,\ldots,T_m \}$ be a set of pairwise non-isomorphic $\alpha\beta$-trees admitting $w$. By Lemma~\ref{number of fixed points}(ii), we would like to choose $r$ so that each $T \in {\mathcal T}$ is hosted by $(\iota, a, a^r)$ for many maps $\iota: V(T) \to \Omega$. As $a$ is fixed, it is beneficial first to examine embeddings of $T_\alpha$ and $T_\beta$ separately. 

We prove results for two kinds of permutations $a$. In the ``generic'' case, all non-trivial cycles of $a$ are long, compared to the $\alpha$- and $\beta$-paths occurring in trees $T$ and in the ``special'' case $\supp(a)$ consists of short  cycles of equal length. We fix a small set $\Lambda \subset \Omega$ (in the application in Section~\ref{sec3}, $|\Lambda| \le 10$) and require that $r$ fixes $\Lambda$ setwise and acts on $\Lambda$ on some prescribed way. The purpose of prescribing the action of $r$ on a small set is to ensure that $w(a,a^r)$ is not trivial. (This trick has already been used in \cite{bbs}.) As the points in $\Lambda$ play a special role, we are only interested in injections $\iota: V(T) \to \Omega \setminus \Lambda$. First, we handle the ``generic'' case. 

\begin{lemma} \label{host talpha}
Let $0<\delta_0< 1/2$ and let $\kappa,\lambda,N$ be positive integers. 
Suppose that $a \in \Sym(\Omega)$ has no cycles of length less than $N$ and that $|\supp(a)|=\delta n$, for some $\delta \in (\delta_0,1-\delta_0)$. Let $\gamma \in \{ \alpha, \beta\}$ and let $T$ be an $\alpha\beta$-tree such that $|V(T)| \le \kappa$, $T$ has no $\gamma$-cycles, and every $\gamma$-path in $T$ has at most $N$ vertices. Let $\Lambda \subseteq \Omega$, $|\Lambda| \le \lambda$, and let 
\[
\mathcal{S}_{\gamma} (T):= \{ \iota: V(T) \to \Omega \setminus \Lambda \mid 
{\mbox{ $T_\gamma$ is hosted by $(\iota,a)$}} \} .
\] 
Then 
\[
|\mathcal{S}_{\gamma} (T)| \ge C(\delta_0,\kappa,\lambda,N,n)\ (1-\delta)^{l_\gamma(T)} \ \delta^{p_\gamma(T)} \ n^{l_\gamma(T)+p_\gamma(T)},
\]
where $C(\delta_0,\kappa,\lambda,N,n)$ is a function with $\lim_{n\to \infty}C(\delta_0,\kappa,\lambda,N,n)=1$. 
\end{lemma}

By $\lim_{n\to \infty}C(\delta_0,\kappa,\lambda,N,n)$, we mean that the variables $\delta_0,\kappa,\lambda$ and $N$ are fixed, and $n$ goes to $\infty$. 

\begin{proof}
Let $v_1,\ldots,v_{l_\gamma(T)}$ be the loops of $T_\gamma$ and let $P_1,\ldots,P_{p_\gamma(T)}$ be the directed paths of $T_\gamma$. We embed the components of $T_\gamma$ into $\Omega \setminus \Lambda$ one-by-one, and estimate $|\mathcal{S}_{\gamma} (T)|$ by the product of the number of possible embeddings at each step.

The vertices $v_i$, for $1 \le i \le l_\gamma(T)$, have to be mapped to fixed points of $a$. As $a$ has at least $n-\delta n - \lambda$ fixed points outside $\Lambda$, $v_1\iota,\ldots, v_{l_\gamma(T)}\iota$ can be chosen in at least $(n-\delta n -\lambda)(n-\delta n-\lambda -1) \cdots (n-\delta n - \lambda - l_\gamma(T)+1) \ge
(n-\delta n - \lambda - \kappa)^{l_\gamma(T)}$ distinct ways. 

Next, we consider the directed paths in $T_\gamma$. For $1 \le i \le p_\gamma(T)$, we write $P_i=(w_{i,0},\ldots,w_{i,c_i})$. Now, once the image of $w_{i,0}$ under $\iota$ is chosen,  in order to guarantee that $\iota$ hosts the path $P_i$, we require that $w_{i,j}\iota=(w_{i,0}\iota)^{a^{j}}$ for each $j=0,\ldots,c_i$. Hence, for each $i\in \{1,\ldots,p_\gamma(T)\}$, the image of $P_i$ under $\iota$ is uniquely determined by $w_{i,0}\iota$. Let $\Delta_i$ be the union of the sets $P_z\iota$, for $z <i$. 
Since by hypothesis $a$ has no cycles of length $N-1$ or shorter and since $T$ has no $\gamma$-path of length greater than $N$, the only requirement for choosing the image of $w_{i,0}$ under $\iota$ is that 
$$w_{i,0}\iota \not\in \bigcup_{j=0}^{c_i}(\Delta_i \cup \Lambda)^{a^{-j}}$$
(since we have to avoid $\Lambda$ and the $\iota$-images of previously mapped $P_z$). A gross overestimate for the size of this union is $(c_i+1)(\lambda+|V(T)|) \le \kappa \lambda + \kappa^2$, and so $w_{i,0}\iota$ can be chosen in at least
$\delta n - \kappa \lambda -\kappa^2$ ways. Summarizing, we obtain 
$$|\mathcal{S}_{\gamma} (T)| \ge (n-\delta n - \lambda - \kappa)^{l_\gamma(T)} \ 
(\delta n - \kappa\lambda - \kappa^2)^{p_\gamma(T)}.$$
By factoring out $\delta$, $(1-\delta)$, and $n$, we get 
\begin{equation*}
|\mathcal{S}_\gamma (T)| \ge C(\delta_0,\kappa,\lambda,N,n) \ (1-\delta)^{l_\gamma(T)}\ \delta^{p_\gamma(T)} \ n^{l_\gamma(T)+p_\gamma(T)},
\end{equation*}
where 
$$C(\delta_0,\kappa,\lambda,N,n)=\left( 1 - \frac{\kappa+\lambda}{\delta_0 n} \right) ^\kappa \left( 1-\frac{\kappa\lambda+\kappa^2}{\delta_0 n} \right) ^\kappa. $$
Clearly, $\lim_{n\to \infty}C(\delta_0,\kappa,\lambda,N,n)=1$. 
\end{proof}

The case of ``special'' permutations $a$ is very similar. 

\begin{lemma}\label{host talpha2}
Let $0<\delta_0< 1/2$ and let $\kappa,\lambda,N$ be positive integers. 
Suppose that  every cycle of $a \in \Sym(\Omega)$ has length $1$ or $N$ and $|\supp(a)|=\delta n$, for some $\delta \in (\delta_0,1-\delta_0)$. Let $\gamma \in \{ \alpha, \beta\}$ and let $T$ be an $\alpha\beta$-tree such that $|V(T)| \le \kappa$, every $\gamma$-cycle in $T$ has $N$ vertices, and every $\gamma$-path in $T$ has at most $N$ vertices. Let $\Lambda \subseteq \Omega$, $|\Lambda| \le \lambda$, and let 
\[
\mathcal{S}_{\gamma} (T):= \{ \iota: V(T) \to \Omega \setminus \Lambda \mid 
{\mbox{ $T_\gamma$ is hosted by $(\iota,a)$}} \} .
\] 
Then 
\[
|\mathcal{S}_{\gamma} (T)| \ge C(\delta_0,\kappa,\lambda,N,n)\ (1-\delta)^{l_\gamma(T)} \ \delta^{p_\gamma(T)} \ n^{l_\gamma(T)+p_\gamma(T)},
\]
where $C(\delta_0,\kappa,\lambda,N,n)$ is a function with $\lim_{n\to \infty}C(\delta_0,\kappa,\lambda,N,n)=1$. 
\end{lemma}

\begin{proof}
We may follow almost verbatim the proof of Lemma~\ref{host talpha}. The only difference is that the list $P_1,\ldots,P_{p_\gamma(T)}$ may also contain $\gamma$-cycles, so we have to change slightly the definition of the vertices $w_{i,0}$. If $P_i$ is a $\gamma$-path then $w_{i,0}$ is the starting vertex of the path as before, while if $P_i$ is a $\gamma$-cycle then $w_{i,0}$ can be chosen as an arbitrary vertex of $P_i$. Since the $\gamma$-cycles of $T_\gamma$ have the same length as the cycles in $\supp(a)$, the rest of the proof goes through without any modification.
\end{proof}

Now we are ready to prove the main result of this section. Let $0<\delta_0<1/2$ and $\kappa,\lambda,N \in \Z^+$ be fixed, and let $w=w_1\cdots w_k$ be a reduced word in $\{ \alpha, \alpha^{-1}, \beta, \beta^{-1} \}$. Suppose further that ${\mathcal T}=\{ T_1,\ldots,T_m \}$ is a set of $\alpha\beta$-trees admitting $w$ and $|V(T)|\le \kappa$ for all $T \in {\mathcal T}$. Let $G=\Sym(n)$ or $\Alt(n)$ be generated by a set $S$ of cardinality at most $n$. We do not have to distinguish the two cases (``generic'' and ``special'') for $a$ anymore, so let $a \in \Sym(\Omega)$ with $|\supp(a)|=\delta n$ for some $\delta \in (\delta_0,1-\delta_0)$ and suppose that either 
\begin{itemize}
\item all non-trivial cycles in $a$ have length at least $N$, none of the 
$\alpha\beta$-trees $T \in {\mathcal T}$ have any cycles, and every $\alpha$- and $\beta$-path in $T$ has at most $N$ vertices; or
\item every cycle of $a$ has length $1$ or $N$, for all $T\in {\mathcal T}$ every $\alpha$- and $\beta$-cycle has $N$ vertices, and every $\alpha$- and $\beta$-path has at most $N$ vertices.
\end{itemize}
Let $\Lambda \subseteq \Omega$, $|\Lambda| \le \lambda$, let $g \in \Sym(\Lambda)$, and let $S_\gamma(T)$ be as in Lemmas~\ref{host talpha} and \ref{host talpha2}. Finally, let an error bound $\varepsilon >0$ be given. We ``collect'' errors in estimates from different sources, so we choose $\varepsilon' < \varepsilon$ such that 
$$\frac{(1-\varepsilon')^3}{1+\varepsilon'}=1-\varepsilon .$$
Moreover, we may assume that $n$ is larger than a bound $n_0(\delta_0,\kappa,\lambda,N,\varepsilon)$ depending only on $\delta_0,\kappa,\lambda,N$, and $\varepsilon$ such that:
\begin{itemize}
\item[(i)] For all $\gamma \in \{ \alpha,\beta \}$ and for all $T \in {\mathcal T}$, 
\begin{equation}
\label{eq:11}
|{\mathcal S}_\gamma(T)|>\left( 1- \varepsilon' \right)  
\ (1-\delta)^{l_\gamma(T)} \ \delta^{p_\gamma(T)} \ n^{l_\gamma(T)+p_\gamma(T)}.
\end{equation}
(Note that this inequality is satisfied by large enough $n$, by Lemmas~\ref{host talpha} and \ref{host talpha2}.)
\item[(ii)] $n^{2(\kappa+\lambda+1)} > 2n^{2(\kappa+\lambda)+1}\log(n^{\kappa+\lambda}/\varepsilon')$.
\end{itemize}
Recall that $\delta_T=(1-\delta)^{l_\alpha(T)+l_\beta(T)}\delta^{p_\alpha(T)+p_\beta(T)}$ and
that $\fixed(T)$ denotes the number of fixed vertices of $(T,w)$. 

\begin{thm}
\label{main estimate}
With the notation of the previous paragraph, there exists $r \in \Sym(\Omega)$ realised by a lazy random walk of length $n^{2(\kappa+\lambda+1)}$ such that $r|_\Lambda = g$ and 
\[
|\fix(w(a,a^r))| > (1-\varepsilon) \ n \ \sum_{j=1}^m \frac{\delta_{T_j} \cdot \fixed(T_j)}{|\Aut(T_j)|}.
\]
\end{thm}

By $r|_\Lambda$ we mean the restriction of the permutation $r$ (considered as a function $r: \Omega \to \Omega$) to $\Lambda$.

\begin{proof}
Let $r$ be realised by a lazy random walk of length $\ell:=n^{2(\kappa+\lambda+1)}$. Our main goal is to give an estimate for the conditional expectation 
$\expec(|\fix(w(a,a^r))| \mid r|_\Lambda = g)$.

Let $T \in {\mathcal T}$ and $\iota \in \mathcal{S}_\alpha(T)$ be arbitrary but fixed. First, we give an estimate for the probability that $T$ is hosted by $(\iota, a, a^r)$. 
By Lemma~\ref{lemmaMarkovChain}, for any $x,y \in (\Omega \setminus \Lambda)_{(|V(T)|)}$, 
$$\prob(x^r=y \wedge r|_\Lambda=g) \ge (1-\varepsilon')\frac{1}{n_{(|\Lambda|+|V(T)|)}} {\mbox{ and }}(1+\varepsilon')\frac{1}{n_{(|\Lambda|)}} \ge \prob(r|_\Lambda=g). $$
Hence, for the conditional probability $\prob(x^r=y \mid  r|_\Lambda=g)$, we have 
\begin{equation}
\label{eq:cond}
\prob(x^r=y \mid  r|_\Lambda=g) \ge \frac{1-\varepsilon'}{1+\varepsilon'}
\frac{1}{(n-|\Lambda|)_{(|V(T)|)}} > \frac{1-\varepsilon'}{1+\varepsilon'}
\frac{1}{n^{|V(T)|}}.
\end{equation}

Since by definition $T_\alpha$ is hosted by $(\iota,a)$, $T$ is hosted by
$(\iota, a, a^r)$ if and only if $T_\beta$ is hosted by $(\iota, a^r)$; this
is equivalent to
\begin{equation*}
\begin{aligned}
& (x\iota)^{a^r}=y\iota \textrm{ for all } (x,y)\in E(T_\beta) \\
\Longleftrightarrow& (x\iota)^{r^{-1}ar} = y\iota  \textrm{ for all }
(x,y)\in E(T_\beta) \\
\Longleftrightarrow& (x\iota)^{r^{-1}a} = (y\iota)^{r^{-1}}  \textrm{ for all
} (x,y)\in E(T_\beta) \\
\Longleftrightarrow& \textrm{ the function } V(T)\to \Omega, x\mapsto (x\iota)^{r^{-1}} \textrm{ is in } \mathcal{S}_\beta(T) \\
\Longleftrightarrow& (V(T)\iota_z)^r=V(T)\iota \textrm{ for some }
\iota_z \in  \mathcal{S}_\beta(T).
\end{aligned}
\end{equation*}
Lemma \ref{number of fixed points} implies that for different $\iota_{z_1},\iota_{z_2} \in 
\mathcal{S}_\beta(T)$, the events $(V(T)\iota_{z_i})^r=V(T)\iota$ are disjoint. Therefore, also using \eqref{eq:11} and  \eqref{eq:cond},
\begin{eqnarray}
\prob(T\mbox{ is hosted by }(\iota,a,a^r) \mid  r|_\Lambda=g) =\sum_{\iota_z \in \mathcal{S}_\beta(T)} \prob\left( (V(T)\iota_{z})^r=V(T)\iota \mid r|_\Lambda=g \right) \notag \\
\label{eq:22} > \frac{(1-\varepsilon')^2}{1+\varepsilon'}
\frac{(1-\delta)^{l_\beta(T)} \ \delta^{p_\beta(T)} \ n^{l_\beta(T)+p_\beta(T)}}{n^{|V(T)|}}.
\end{eqnarray}
Next, by Lemma~\ref{number of fixed points}(ii) and \eqref{eq:11}, \eqref{eq:22},
\begin{eqnarray*}
\expec(|\fix(w(a,a^r))| \ourmid r|_\Lambda=g) \geq \sum_{j=1}^m 
\sum_{\iota \in \mathcal{S}_\alpha(T_j)} \prob(T_j\mbox{ is hosted by }(\iota,a,a^r) \ourmid  r|_\Lambda=g) \frac{\fixed(T_j)}{|\Aut(T_j)|} \\
> \sum_{j=1}^m \frac{(1-\varepsilon')^3}{1+\varepsilon'} 
\frac{(1-\delta)^{l_\alpha(T_j)+l_\beta(T_j)} \ \delta^{p_\alpha(T_j)+p_\beta(T_j)} \ n^{l_\alpha(T_j)+l_\beta(T_j)+p_\alpha(T_j)+p_\beta(T_j)}}{n^{|V(T_j)|}}
\frac{\fixed(T_j)}{|\Aut(T_j)|}.
\end{eqnarray*}
Finally, by Lemma~\ref{tree}, 
$ l_\alpha(T_j)+l_\beta(T_j)+p_\alpha(T_j)+p_\beta(T_j) = |V(T_j)|+1$, yielding
\[
\expec(|\fix(w(a,a^r))| \ourmid r|_\Lambda=g) >  
(1-\varepsilon) \ n \ \sum_{j=1}^m \frac{\delta_{T_j} \cdot \fixed(T_j)}{|\Aut(T_j)|}.
\]
To finish the proof of the theorem, we simply take $r$ that gives at least the expected number of fixed points. 
\end{proof}

\section{Proof of Theorem~\ref{main1}}\label{sec3}

In this section we apply Theorem~\ref{main estimate} to prove Theorem~\ref{main1}. We start with two technical lemmas.

\begin{lemma}\label{silly} 
Let $m$ be an integer and take $g,h\in \Sym(m)$. In each of the following cases, $[h,(h^g)^{-1}][h,h^g]$ contains a $7$-cycle:
\begin{enumerate}
\item\label{i} $m\geq 7$, $h=(1,2,3,\dots, m)$, $g$ contains the cycle $(1,3,m)$ and fixes the points $2,4,5,6,m-3,m-2,m-1$;
\item\label{ii} $m=7$, $h=(1,2,3,4,5)(6)(7)$, $g=(1,6)(3,7)(2)(4)(5)$;
\item\label{iii} $m=7$, $h=(1,2,3)(4,5,6)(7)$, $g=(1,7,2,4)(3)(5)(6)$;
\item\label{iv} $m=7$, $h=(1,2)(3,4)(5,6)(7)$, $g=(1,5,7,2,3)(4)(6)$.
\end{enumerate}
\end{lemma}
\begin{proof}
To show $(1)$, we note simply that $(1,m,5,3,m-1,4,2)$ is a $7$-cycle of $[h,(h^g)^{-1}][h,h^g]$. Parts~$(2),(3)$ and $(4)$ follow easily.
\end{proof}

Note that in each case in Lemma~\ref{silly} we prescribe the action of $g$ on at most 10 points. 

For $0<\delta <1$, define 
\begin{equation}\label{bigfeller}
\begin{aligned}
f(\delta):=
&(1 - \delta)^2 + (1 - \delta)^2 \delta + (1 - \delta)^3 \delta + 4 (1 - \delta)^4 \delta^2 + 2 (1 - \delta)^2 \delta^3 \\ &+
 3 (1 - \delta)^5 \delta^3 + 10 (1 - \delta)^9 \delta^4 + 26 (1 - \delta)^7 \delta^5 + 20 (1 - \delta)^8 \delta^5 \\ &+
 6 (1 - \delta)^5 \delta^6 + 16 (1 - \delta)^6 \delta^6 + 40 (1 - \delta)^8 \delta^6 + 3 (1 - \delta)^4 \delta^7 \\ &+
 8 (1 - \delta)^6 \delta^7 + 20 (1 - \delta)^7 \delta^7 + 10 (1 - \delta)^8 \delta^9 + 20 (1 - \delta)^7 \delta^{10} \\ &+
 10 (1 - \delta)^6 \delta^{11} + 15 (1 - \delta)^7 \delta^{11}.
\end{aligned}
\end{equation}

\begin{lemma}\label{festimates}\
\begin{enumerate}
\item The function $\delta \mapsto 1-0.999f(\delta)$ is monotone increasing on the interval $(0,1)$.
\item The equation $0.999 f(\delta)=1-\delta$ has a unique solution in $(0,1)$. Up to six significant digits, the solution is $\delta=0.632599$.
\item\label{three} Starting with $\delta =0.63$, nine iterations of the function $\delta \mapsto 1-0.999f(\delta)$ reach a value less than $0.326$. 
\end{enumerate}
\end{lemma}

\begin{proof}
All three results can be established using an algebra package such as Sage~\cite{sage}.
\end{proof}

We are now ready to prove Theorem~\ref{main1}.

\begin{proof}[Proof of Theorem~{\em \ref{main1}}] Let $a$ be an element of $S$ with $|\supp(a)|<Cn=0.63n$. We shall apply the results of Section~\ref{sec: machine} for the word $w=w_0^{60}$, where $w_0=w_0(\alpha,\beta)=[\alpha,\beta^{-1}][\alpha,\beta]$. 

The proof splits into a number of cases, according to the order of $a$. Let $|a|=2^{e_1}3^{e_2}e_3$, with $e_3$ coprime to $6$. Note that $2^{e_1} \le n$ and $3^{e_2} \le n$. Suppose first that $e_3>1$. Then 
$a^{2^{e_1}3^{e_2}}$ is non-trivial, and we may replace $a$ with $a^{2^{e_1}3^{e_2}}$ and assume that the order of $a$ is coprime to both $2$ and $3$ (note that $a^{2^{e_1}3^{e_2}}\in (S\cup S^{-1} \cup \{1\})^{n^2}$).

In Tables~\ref{table: w itself only}, \ref{table:2cycles}, \ref{table:3cycles}, \ref{table:4cycles}, and \ref{table:5cycles}, we present a family ${\mathcal T}$ of $\alpha\beta$-trees admitting $w_0^{60}$. These $\alpha\beta$-trees were obtained with the help of a computer; note that our conventions for representing them are outlined in Subsection~\ref{s: words and graphs}. Note too that if an $\alpha\beta$-graph admits $w_0^{e}$ then it also admits $w_0^{ek}$ for any positive integer $k$. The $\alpha\beta$-trees in these tables are pairwise non-isomorphic, do not have non-identity automorphisms, contain at most $16$ vertices, and have all $\alpha$-paths and $\beta$-paths of length at most $4$. For this family, 
\[
\sum_{T \in {\mathcal T}} \frac{\delta_{T} \cdot \fixed(T)}{|\Aut(T)|}=f(\delta),
\]
where $f(\delta)$ is as in \eqref{bigfeller}. 
If $a$ has a cycle of length $m\geq 7$ then relabel the letters of $\Omega$ so that this cycle is equal to $(1,\dots, m)$, define $\Lambda=\{1,2,3,4,5,6,m-3,m-2,m-1,m\}$ and let $g=(1,3,m)(2)(4)(5)(6)(m-3)(m-2)(m-1)\in\Sym(\Lambda)$. If this is not the case, then we must have $e_3=5$; in this case relabel the letters of $\Omega$ so that $a$ contains the cycle $(1,2,3,4,5)$ and $a$ fixes both $6$ and $7$. Then define $\Lambda=\{1,\dots, 7\}$ and let $g=(1,6)(3,7)(2)(4)(5)$.

By Lemma~\ref{reduce genset}, we may suppose that $|S| \le n$. We apply Theorem~\ref{main estimate} with $\delta_0=0.3$, $w=w_0^{60}$, ${\mathcal T}$, $N=4$, $\lambda=10$, $g$, $\kappa=16$, and $\varepsilon=0.001$. We obtain $r \in (S \cup S^{-1}\cup\{1\})^{n^{54}}$ such that $|\fix(w(a,a^r))| \ge 0.999 f(\delta)$. Note that $w(a,a^r)\in (S \cup S^{-1}\cup\{1\})^{n^{54}+n^2}$ and, of course, $n^{54}+n^2=O(n^{54})$. By Lemma~\ref{silly}, (\ref{i}) and (\ref{ii}), $w(a,a^r)$ is non-trivial because it contains a $7$-cycle. Replacing $a$ by $w(a,a^r)$ and using the same procedure as above, Lemma~\ref{festimates}, part (\ref{three}), implies that in at most nine iterations we obtain a permutation $a'$ with support size less than $0.326 n$. Each iteration increases the word length by a factor $O(n^2)$, as we may have to raise the input permutation to a suitable power to eliminate $2$ and $3$ from the cycle lengths, while conjugating by (a new) $r$ and substituting into the word $w$ contributes only constant multipliers to the word length. Hence $a'$ is a word in $S$ of length $O(n^{70})$ and, by Theorem~\ref{t: bbs}, $\diam(\Gamma(G,S))=O(n^{78})$.  

Suppose next that $e_3=1$, that is, $a$ has order $2^{e_1}3^{e_2}$. If $e_1>0$, let $k=2^{e_1-1}3^{e_2}$, otherwise let $k=3^{e_2-1}$. Then $k<n^2$, $a^k$ has order 2 or 3, and $a^k\in (S\cup S^{-1}\cup\{1\})^{n^2}$. 

In Tables~\ref{table: prime2} and \ref{table: prime3}, we present two families of $\alpha\beta$-trees admitting $w=w_0^{60}$, represented as before. The $\alpha\beta$-trees in Table~\ref{table: prime2} (resp. Table~\ref{table: prime3}) are pairwise non-isomorphic, contain at most 10 vertices, have all $\alpha$-cycles and $\beta$-cycles of length $2$ (resp. $3$), and have all $\alpha$-paths and $\beta$-paths of length at most $1$ (resp. $2$). 

In each cell of Tables~\ref{table: prime2} and \ref{table: prime3}, we have written $\Aut=k$ to mean that the automorphism group of the corresponding $\alpha\beta$-tree $T$ has size $k$. We define ${\mathcal T}$ to be the set of $\alpha\beta$-trees in Table~\ref{table: prime2} (resp. Table~\ref{table: prime3}) when $a$ has order $2$ (resp. has order $3$). For these families, 
$$\sum_{T \in {\mathcal T}} \frac{\delta_{T} \cdot \fixed(T)}{|\Aut(T)|}=h(\delta)$$
where
\begin{equation*}
h(\delta)=\left\{\begin{array}{ll}
(1-\delta)^2(1+2\delta+3\delta^2+ 4\delta^3+5\delta^4+6\delta^5)     &\textrm{if }a\textrm{ has order 2}; \\
& \\
(1 - \delta)^2 + \delta (1 - \delta)^2 + \delta (1 - \delta)^3 +2 \delta^3 (1 - \delta)^2 & \\
+4\delta^2 (1 - \delta)^4  + 3\delta^3 (1 - \delta)^5  + 12 \delta^{7} (1 - \delta)^4   &\textrm{if }a\textrm{ has order 3.} \\
+6\delta^6 (1 - \delta)^4+ \delta^4 (1 - \delta)^6   & \\
                 \end{array}\right.
\end{equation*}
 
Define $\Lambda=\{1,\dots, 7\}$. If $a$ has order $2$ label the elements of $\Omega$ so that $a|_\Lambda=(1,2)(3,4)(5,6)(7)$; then define $g=(1,5,7,2,3)(4)(6)$. If $a$ has order $3$ label the elements of $\Omega$ so that $a|_\Lambda=(1,2,3)(4,5,6)(7)$; then define $g=(1,7,2,4)(3)(5)(6)$.

By Lemma~\ref{reduce genset}, we may suppose that $|S| \le n$. We define $N$ to equal $1$ (resp. $2$) when $a$ has order $2$ (resp. has order $3$). We apply Theorem~\ref{main estimate} with $\delta_0=0.3$, $w=w_0^{60}$, ${\mathcal T}$, $N$, $\lambda=7$, $g$, $\kappa=10$, and $\varepsilon=0.001$. We obtain $r \in (S \cup S^{-1}\cup\{1\})^{n^{36}}$ such that $|\fix(w(a,a^r))| \ge 0.999 h(\delta)$. 

Using Sage \cite{sage} it is easy to check that the function $\delta \mapsto 1-0.999h(\delta)$ is monotone increasing on the interval $(0,1)$. Furthermore, for $\delta\leq 0.63$, we have $0.999h(\delta)>0.374$ and so $|\supp(w(a,a^r)|<0.626$. Note that $w(a,a^r)\in (S\cup S^{-1})^{O(n^{36})}$ and, by Lemma~\ref{silly}, (\ref{iii}) and (\ref{iv}),  $w(a,a^r)$ contains a $7$-cycle. 

We now run the first part of the argument using the element $w(a,a^r)$ instead of $a$ as our initial element of small support. After one iteration we obtain an element $a'\in (S\cup S^{-1}\cup \{1\})^{n^{2}O(n^{36})+n^{54}}$ with support of size less than $1-0.999f(\delta)$. Iterating as before, we obtain that
 $\diam(\Gamma(G,S))=O(n^{78})$.  
\end{proof}

\section{Improving Theorem~\ref{main1}}\label{sec: closing}

It should be clear to the reader that Theorem \ref{main1} is not optimal. In particular we prove Theorem \ref{main1} with respect to a particular word, $w=w_0^{60}$ where $w_0=[\alpha,\beta^{-1}][\alpha,\beta]$; it is this word that yields the value $C=0.63$. How might one go about improving this value?

The most obvious way of improving Theorem~\ref{main1} is via an appeal to higher powers. Consider $w_k=w_0^k$ where $k$ is any multiple of 60. All of the trees in Tables~\ref{table: w itself only}, \ref{table:2cycles}, \ref{table:3cycles}, \ref{table:4cycles} and \ref{table:5cycles} admit $w_k$ and, for suitable choices of $k$, there will be yet more trees to consider. This will inevitably result in an increase for the value of $C$.

There is a limit to the improvement that such a strategy might yield and we briefly explain why this is the case. Fix a word $v_0$, let $k$ be some positive integer, and define the word $v_k=v_0^k$. Let $\mathcal{T}$ be a set of $\alpha\beta$-trees admitting $v_k$ and consider the sum
\begin{equation}\label{e: explanation}
g_k(\delta) = \sum_{T \in {\mathcal T}} \frac{\delta_{T} \cdot \fixed(T)}{|\Aut(T)|}.
\end{equation}
The main result of Section~\ref{sec: machine}, Theorem~\ref{main estimate}, gives a lower bound for $|\fix(w(a,a^r))|$ in terms of $g_k(\delta)$, $\varepsilon$ and $n$; here $a$ is an element of support equal to $\delta n$ and $r$ is some short word. In particular, if the value of $g_k(\delta)$ exceeds the value of $1-\delta$ (by a sufficient margin in terms of $\varepsilon$) then we obtain an element of smaller support than $a$.

The advantage of considering higher powers of the word $v_0$ is exhibited in Theorem~\ref{main estimate} by noting that if, say, $k$ doubles, then new $\alpha\beta$-trees may be added to $\mathcal{T}$, thereby increasing the value of $g_k(\delta)$ for $\delta\in(0,1)$. 

Using methods different to those in this paper, the authors have developed a method to show that, in the ``generic'' case (see Section~\ref{sec: machine} for an explanation of what we mean by this), there is a number $\delta_0\in(0,1)$ such that
$$\limsup_{k\to\infty} g_k(\delta)<1-\delta$$
whenever $\delta>\delta_0$. It is important to note that the number $\delta_0$ depends only on the word $v_0$. Furthermore the number $\delta_0$ corresponds to the unique all-positive solution to a certain system of polynomial equations with rational coefficients and this system depends only on the word $v_0$. 

In the case $v_0=w_0=[\alpha,\beta^{-1}][\alpha,\beta]$, our method shows that $\delta_0 \approx 0.64242$. 
One can see, then, that the value for $C$ given in Theorem~\ref{main1} is close to optimal when it comes to powers of the word $w_0$.

In another direction one might hope to improve Theorem~\ref{main1} using an entirely
different choice of word $w_0$. 
With this in mind we undertook an exhaustive search of words of length
at most 20 which were {\em balanced} (i.e. $\alpha,\alpha^{-1}, \beta$ and $\beta^{-1}$
all occur the same number of times) and in which $\alpha^{\pm 1}$ and
$\beta^{\pm 1}$ occur alternately. For each such word $w_1$, we performed the 
following computer experiment. For a variety of values $n$ in the range $10^4 \le n \le 10^5$ and $\delta$ in the range $0.55 \le \delta \le 0.70$, we took a
random permutation $a \le S_n$ with $|\supp(a)|=\delta n$, we constructed $b$ as a random conjugate of $a$, and counted how many points occur in cycles of length at most $6$ in the permutation $w_1(a,b)$. The highest counts occured for $w_0=[\alpha,\beta^{-1}][\alpha,\beta]$ and for related words (like the cyclic permutations of $w_0$ or $w_0^2$)
and this is the reason for our use of the word $w_0$ in the preceding proof. 

We also carried out a non-exhaustive investigation into words that were either non-balanced or non-alternating. In every case, for a word $w_1$ of this kind, and for $a$ and $b$ as above, the computer tests suggested that permutations of the form $w_1(a,b)$ tended to have a smaller number of points in short cycles than $w_0(a,b)$.

\newcommand{\etalchar}[1]{$^{#1}$}
\bibitem[S{\etalchar{+}}09]{sage}
W.\thinspace{}A. Stein et~al., \emph{{S}age {M}athematics {S}oftware ({V}ersion
  4.8)}, The Sage Development Team, 2012, {\tt http://www.sagemath.org}


\bibitem{whiston}J.~Whiston, Maximal independent generating sets of the symmetric group,
\textit{J. Algebra} \textbf{232} (2000), no. 1, 255--268.

\bibitem{wielandt2}H.~Wielandt, Absch\"atzungen f\"ur den Grad einer Permutationsgruppe von vorgeschriebenem Transitivit\"atsgrad, Dissertation, (1934) Berlin.

\begin{table}[H]
\caption{$\alpha\beta$-trees admitting $w^2$  but not $w$}\label{table:2cycles}
\begin{tabular}{|c|c|}
\hline
\parbox[b]{\sizetwo}{
\resizebox{\sizetwo}{!}{\input{2cycles/tikz/g1.tex}} \dd{2}{4}}
&
\parbox[b]{\sizetwo}{
\resizebox{\sizetwo}{!}{\input{2cycles/tikz/g2.tex}} \dd{2}{4}}
\\ \hline 
\end{tabular} \medskip
\end{table}

%
%

\begin{table}[H]
\caption{$\alpha\beta$-trees admitting $w^3$ but not $w$}\label{table:3cycles}
\begin{tabular}{|c|c|c|c|}
\hline
\parbox[b]{\sizethree}{
\resizebox{\sizethree}{!}{ \input{3cycles/tikz/g4.tex}} \dd{3}{5}}
&
\parbox[b]{\sizefour}{
\resizebox{\sizefour}{!}{ \input{3cycles/tikz/g1.tex}} \dd{7}{4}}
&
\parbox[b]{\sizefour}{
\resizebox{\sizefour}{!}{ \input{3cycles/tikz/g2.tex}} \dd{6}{5}}
&
\parbox[b]{\sizefour}{
\resizebox{\sizefour}{!}{ \input{3cycles/tikz/g3.tex}} \dd{6}{5}}
\\ \hline 
\end{tabular} \medskip
\end{table}

%
%

\begin{table}[H]
\caption{$\alpha\beta$-trees admitting $w^4$ but not $w^2$}\label{table:4cycles}
\begin{tabular}{|c|c|c|c|c|}
\hline
\parbox[b]{\sizefour}{
\resizebox{\sizefour}{!}{ \input{4cycles/tikz/g1.tex}} \dd{6}{6}}  
&
\parbox[b]{\sizethree}{
\resizebox{\sizethree}{!}{ \input{4cycles/tikz/g2.tex}} \dd{5}{7}} 
&
\parbox[b]{\sizefour}{
\resizebox{\sizefour}{!}{ \input{4cycles/tikz/g3.tex}} \dd{6}{6}} 
&
\parbox[b]{\sizefour}{
\resizebox{\sizefour}{!}{ \input{4cycles/tikz/g4.tex}} \dd{6}{6}} 
&
\parbox[b]{\sizefour}{
\resizebox{\sizefour}{!}{ \input{4cycles/tikz/g5.tex}} \dd{6}{6}} 
\\ \hline
\parbox[b]{\sizefour}{
\resizebox{\sizefour}{!}{ \input{4cycles/tikz/g6.tex}} \dd{5}{7}} 
&
\parbox[b]{\sizefour}{
\resizebox{\sizefour}{!}{ \input{4cycles/tikz/g7.tex}} \dd{5}{7}} 
&
\parbox[b]{\sizefour}{
\resizebox{\sizefour}{!}{ \input{4cycles/tikz/g8.tex}} \dd{5}{7}} 
&
\parbox[b]{\sizefive}{
\resizebox{\sizefive}{!}{ \input{4cycles/tikz/g9.tex}} \dd{7}{6}} 
&
\parbox[b]{\sizefive}{
\resizebox{\sizefive}{!}{ \input{4cycles/tikz/g10.tex}} \dd{7}{6}} 
\\ \hline 
\end{tabular}
\end{table}

%
%

\begin{longtable}{|c|c|c|c|c|}
\caption{$\alpha\beta$-trees admitting $w^5$ but not $w$}\label{table:5cycles}\\
\hline
\parbox[b]{\sizefive}{
\resizebox{\sizefive}{!}{ \input{5cycles/tikz/g1.tex}} \dd{5}{7}} 
&  
\parbox[b]{\sizefive}{
\resizebox{\sizefour}{!}{ \input{5cycles/tikz/g2.tex}} \dd{5}{7}}
& 
\parbox[b]{\sizefive}{
\resizebox{\sizefour}{!}{ \input{5cycles/tikz/g3.tex}} \dd{5}{8}}
& 
\parbox[b]{\sizefive}{
\resizebox{\sizefour}{!}{ \input{5cycles/tikz/g4.tex}} \dd{5}{8}}
\\ \hline 
\parbox[b]{\sizefive}{
\resizebox{\sizefour}{!}{ \input{5cycles/tikz/g5.tex}} \dd{5}{8}}
& 
\parbox[b]{\sizefive}{
\resizebox{\sizefour}{!}{ \input{5cycles/tikz/g6.tex}} \dd{5}{8}}
& 
\parbox[b]{\sizefive}{
\resizebox{\sizethree}{!}{ \input{5cycles/tikz/g7.tex}} \dd{4}{9}}
& 
\parbox[b]{\sizefive}{
\resizebox{\sizefour}{!}{ \input{5cycles/tikz/g8.tex}} \dd{4}{9}}
\\ \hline 
\parbox[b]{\sizefive}{
\resizebox{\sizefive}{!}{ \input{5cycles/tikz/g9.tex}} \dd{7}{7}}
& 
\parbox[b]{\sizefive}{
\resizebox{\sizefive}{!}{ \input{5cycles/tikz/g10.tex}} \dd{7}{7}}
& 
\parbox[b]{\sizefive}{
\resizebox{\sizefive}{!}{ \input{5cycles/tikz/g11.tex}} \dd{7}{7}}
& 
\parbox[b]{\sizefive}{
\resizebox{\sizefive}{!}{ \input{5cycles/tikz/g12.tex}} \dd{6}{8}}
\\ \hline 
\parbox[b]{\sizefive}{
\resizebox{\sizefour}{!}{ \input{5cycles/tikz/g13.tex}} \dd{6}{8}}
& 
\parbox[b]{\sizefive}{
\resizebox{\sizefive}{!}{ \input{5cycles/tikz/g14.tex}} \dd{7}{7}}
& 
\parbox[b]{\sizefive}{
\resizebox{\sizefive}{!}{ \input{5cycles/tikz/g15.tex}} \dd{6}{8}}
& 
\parbox[b]{\sizefive}{
\resizebox{\sizefive}{!}{ \input{5cycles/tikz/g16.tex}} \dd{6}{8}}
\\ \hline 
\parbox[b]{\sizefive}{
\resizebox{\sizefive}{!}{ \input{5cycles/tikz/g17.tex}} \dd{6}{8}}
& 
\parbox[b]{\sizefive}{
\resizebox{\sizefive}{!}{ \input{5cycles/tikz/g18.tex}} \dd{6}{8}}
& 
\parbox[b]{\sizefive}{
\resizebox{\sizefive}{!}{ \input{5cycles/tikz/g19.tex}} \dd{6}{8}}
& 
\parbox[b]{\sizefive}{
\resizebox{\sizefive}{!}{ \input{5cycles/tikz/g20.tex}} \dd{6}{8}}
\\ \hline 
\parbox[b]{\sizefive}{
\resizebox{\sizefive}{!}{ \input{5cycles/tikz/g21.tex}} \dd{10}{7}}
& 
\parbox[b]{\sizefive}{
\resizebox{\sizefive}{!}{ \input{5cycles/tikz/g22.tex}} \dd{11}{6}}
& 
\parbox[b]{\sizefive}{
\resizebox{\sizefive}{!}{ \input{5cycles/tikz/g23.tex}} \dd{10}{7}}
& 
\parbox[b]{\sizefive}{
\resizebox{\sizefive}{!}{ \input{5cycles/tikz/g24.tex}} \dd{11}{6}}
\\ \hline 
\parbox[b]{\sizefive}{
\resizebox{\sizefive}{!}{ \input{5cycles/tikz/g25.tex}} \dd{10}{7}}
& 
\parbox[b]{\sizefive}{
\resizebox{\sizefive}{!}{ \input{5cycles/tikz/g26.tex}} \dd{10}{7}}
& 
\parbox[b]{\sizefive}{
\resizebox{\sizefive}{!}{ \input{5cycles/tikz/g27.tex}} \dd{9}{8}}
& 
\parbox[b]{\sizefive}{
\resizebox{\sizefive}{!}{ \input{5cycles/tikz/g28.tex}} \dd{9}{8}}
\\ \hline 
\parbox[b]{\sizefive}{
\resizebox{\sizefive}{!}{ \input{5cycles/tikz/g29.tex}} \dd{11}{7}}
& 
\parbox[b]{\sizefive}{
\resizebox{\sizefive}{!}{ \input{5cycles/tikz/g30.tex}} \dd{11}{7}}
& 
\parbox[b]{\sizefive}{
\resizebox{\sizefive}{!}{ \input{5cycles/tikz/g31.tex}} \dd{11}{7}}
& \\ \hline 
\end{longtable}

\begin{longtable}{|c|c|c|c|c|}
\caption{$\alpha\beta$-trees for an involution}\label{table: prime2}\\
\hline
\parbox[b]{30mm}{
\resizebox{!}{3mm}{
\begin{tikzpicture}[>=latex,line join=bevel,]
\tikzstyle{vx0}=[circle, ball color=blue!80!black, inner sep=5bp];\tikzstyle{vx1}=[star, ball color=red, inner sep=4bp];\tikzstyle{ed0}=[>=triangle 45,ultra thick,dashed,red];\tikzstyle{ed1}=[>=triangle 45,thick,blue!80!black];\begin{pgfscope}[scale=0.5]%
\node (0) at (27bp,18bp) [vx1] {};
\end{pgfscope}%
\end{tikzpicture}

} \\ \dd{0}{2} \\ $\mathsf{Aut}= 1$}
&
\parbox[b]{30mm}{
\resizebox{!}{9mm}{\input{order2/g1.tex}} \\ \dd{1}{2} \\ $\mathsf{Aut}= 2$}
&
\parbox[b]{30mm}{
\resizebox{!}{9mm}{\input{order2/g2.tex}} \\ \dd{1}{2} \\ $\mathsf{Aut}= 2$}
&
\parbox[b]{30mm}{
\resizebox{!}{18mm}{\input{order2/g4.tex}} \\ \dd{2}{2} \\ $\mathsf{Aut}= 1$}
&
\parbox[b]{30mm}{
\resizebox{!}{27mm}{\input{order2/g5.tex}} \\ \dd{3}{2} \\ $\mathsf{Aut}= 2$}
\\ \hline
\parbox[b]{30mm}{
\resizebox{!}{27mm}{\input{order2/g6.tex}} \\ \dd{3}{2} \\ $\mathsf{Aut}= 2$}
&
\parbox[b]{30mm}{
\resizebox{!}{36mm}{\input{order2/g7.tex}} \\ \dd{4}{2} \\ $\mathsf{Aut}= 1$}
&
\parbox[b]{30mm}{
\resizebox{!}{45mm}{\input{order2/g8.tex}} \\ \dd{5}{2} \\ $\mathsf{Aut}= 2$}
&
\parbox[b]{30mm}{
\resizebox{!}{45mm}{\input{order2/g9.tex}} \\ \dd{5}{2} \\ $\mathsf{Aut}= 2$}
&
\\ \hline
\end{longtable}

%
\begin{longtable}{|c|c|c|c|}
\caption{$\alpha\beta$-graphs for an element of order $3$}\label{table: prime3}\\
\hline
\parbox[b]{30mm}{\begin{tikzpicture}[>=latex,line join=bevel,]
\tikzstyle{vx0}=[circle, ball color=blue!80!black, inner sep=5bp];\tikzstyle{vx1}=[star, ball color=red, inner sep=4bp];\tikzstyle{ed0}=[>=triangle 45,ultra thick,red,dashed];\tikzstyle{ed1}=[>=triangle 45,ultra thick,thick,blue!80!black];\begin{pgfscope}[scale=0.5]%
  \node (2) at (27bp,90bp) [vx1] {};
\end{pgfscope}%
\end{tikzpicture} \\ \dd{0}{2} \\ $\mathsf{Aut}= 1$}
& 
\parbox[b]{30mm}{
\resizebox{!}{9mm}{\begin{tikzpicture}[>=latex,line join=bevel,]
\tikzstyle{vx0}=[circle, ball color=blue!80!black, inner sep=5bp];\tikzstyle{vx1}=[star, ball color=red, inner sep=4bp];\tikzstyle{ed0}=[>=triangle 45,ultra thick,red,dashed];\tikzstyle{ed1}=[>=triangle 45,ultra thick,thick,blue!80!black];\begin{pgfscope}[scale=0.5]%
\node (1) at (27bp,18bp) [vx1] {};
  \node (0) at (27bp,90bp) [vx0] {};
  \draw [->,ed0] (0) ..controls (27bp,64.131bp) and (27bp,54.974bp)  .. (1);
\end{pgfscope}%
\end{tikzpicture}} \\ \dd{1}{2} \\ $\mathsf{Aut}= 1$}
&
\parbox[b]{30mm}{
\resizebox{!}{18mm}{\begin{tikzpicture}[>=latex,line join=bevel,]
\tikzstyle{vx0}=[circle, ball color=blue!80!black, inner sep=5bp];\tikzstyle{vx1}=[star, ball color=red, inner sep=4bp];\tikzstyle{ed0}=[>=triangle 45,ultra thick,red,dashed];\tikzstyle{ed1}=[>=triangle 45,ultra thick,thick,blue!80!black];\begin{pgfscope}[scale=0.5]%
\node (1) at (27bp,162bp) [vx0] {};
  \node (0) at (27bp,18bp) [vx0] {};
  \node (2) at (27bp,90bp) [vx1] {};
  \draw [->,ed1] (1) ..controls (27bp,136.13bp) and (27bp,126.97bp)  .. (2);
  \draw [->,ed1] (2) ..controls (27bp,64.131bp) and (27bp,54.974bp)  .. (0);
\end{pgfscope}%
\end{tikzpicture}} \\ \dd{1}{3} \\ $\mathsf{Aut}= 1$}
&
\parbox[b]{30mm}{
\resizebox{!}{18mm}{\input{order3/g5.tex}} \\ \dd{3}{2} \\ $\mathsf{Aut}= 1$}
\\ \hline
\parbox[b]{30mm}{
\resizebox{!}{18mm}{\input{order3/g6.tex}} \\ \dd{3}{2} \\ $\mathsf{Aut}= 1$}
&
\parbox[b]{30mm}{
\resizebox{!}{36mm}{
\begin{tikzpicture}[>=latex,line join=bevel,]
\tikzstyle{vx0}=[circle, ball color=blue!80!black, inner sep=5bp];\tikzstyle{vx1}=[star, ball color=red, inner sep=4bp];\tikzstyle{ed0}=[>=triangle 45,ultra thick,red,dashed];\tikzstyle{ed1}=[>=triangle 45,ultra thick,thick,blue!80!black];\begin{pgfscope}[scale=0.5]%
\node (1) at (82bp,90bp) [vx0] {};
  \node (3) at (27bp,234bp) [vx0] {};
  \node (5) at (54bp,306bp) [vx0] {};
  \node (4) at (54bp,18bp) [vx1] {};
  \node (6) at (54bp,162bp) [vx1] {};
  \draw [->,ed0] (3) ..controls (36.748bp,208.01bp) and (40.462bp,198.1bp)  .. (6);
  \draw [->,ed1] (6) ..controls (63.987bp,136.32bp) and (67.894bp,126.27bp)  .. (1);
  \draw [->,ed1] (4) ..controls (48.891bp,46.402bp) and (46.891bp,59.902bp)  .. (46bp,72bp) .. controls (44.824bp,87.957bp) and (44.824bp,92.043bp)  .. (46bp,108bp) .. controls (46.627bp,116.51bp) and (47.802bp,125.71bp)  .. (6);
  \draw [->,ed0] (6) ..controls (59.747bp,190.38bp) and (61.997bp,203.88bp)  .. (63bp,216bp) .. controls (64.32bp,231.95bp) and (64.32bp,236.05bp)  .. (63bp,252bp) .. controls (62.294bp,260.52bp) and (60.973bp,269.73bp)  .. (5);
  \draw [->,ed1] (1) ..controls (72.013bp,64.319bp) and (68.106bp,54.274bp)  .. (4);
  \draw [->,ed0] (5) ..controls (44.252bp,280.01bp) and (40.538bp,270.1bp)  .. (3);
\end{pgfscope}%
\end{tikzpicture}} \\ \dd{2}{4} \\ $\mathsf{Aut}= 1$}
&
\parbox[b]{30mm}{
\resizebox{!}{36mm}{
\begin{tikzpicture}[>=latex,line join=bevel,]
\tikzstyle{vx0}=[circle, ball color=blue!80!black, inner sep=5bp];\tikzstyle{vx1}=[star, ball color=red, inner sep=4bp];\tikzstyle{ed0}=[>=triangle 45,ultra thick,red,dashed];\tikzstyle{ed1}=[>=triangle 45,ultra thick,thick,blue!80!black];\begin{pgfscope}[scale=0.5]%
\node (1) at (82bp,90bp) [vx1] {};
  \node (3) at (27bp,234bp) [vx0] {};
  \node (5) at (54bp,306bp) [vx1] {};
  \node (4) at (54bp,18bp) [vx0] {};
  \node (6) at (54bp,162bp) [vx0] {};
  \draw [->,ed0] (3) ..controls (36.748bp,208.01bp) and (40.462bp,198.1bp)  .. (6);
  \draw [->,ed1] (6) ..controls (63.987bp,136.32bp) and (67.894bp,126.27bp)  .. (1);
  \draw [->,ed1] (4) ..controls (48.891bp,46.402bp) and (46.891bp,59.902bp)  .. (46bp,72bp) .. controls (44.824bp,87.957bp) and (44.824bp,92.043bp)  .. (46bp,108bp) .. controls (46.627bp,116.51bp) and (47.802bp,125.71bp)  .. (6);
  \draw [->,ed0] (6) ..controls (59.747bp,190.38bp) and (61.997bp,203.88bp)  .. (63bp,216bp) .. controls (64.32bp,231.95bp) and (64.32bp,236.05bp)  .. (63bp,252bp) .. controls (62.294bp,260.52bp) and (60.973bp,269.73bp)  .. (5);
  \draw [->,ed1] (1) ..controls (72.013bp,64.319bp) and (68.106bp,54.274bp)  .. (4);
  \draw [->,ed0] (5) ..controls (44.252bp,280.01bp) and (40.538bp,270.1bp)  .. (3);
\end{pgfscope}%
\end{tikzpicture}} \\ \dd{2}{4} \\ $\mathsf{Aut}= 1$}
&
\parbox[b]{30mm}{
\resizebox{!}{54mm}{
\begin{tikzpicture}[>=latex,line join=bevel,]
\tikzstyle{vx0}=[circle, ball color=blue!80!black, inner sep=5bp];\tikzstyle{vx1}=[star, ball color=red, inner sep=4bp];\tikzstyle{ed0}=[>=triangle 45,ultra thick,red,dashed];\tikzstyle{ed1}=[>=triangle 45,ultra thick,thick,blue!80!black];\begin{pgfscope}[scale=0.5]%
\node (1) at (82bp,90bp) [vx0] {};
  \node (0) at (54bp,450bp) [vx1] {};
  \node (3) at (27bp,234bp) [vx1] {};
  \node (2) at (27bp,378bp) [vx0] {};
  \node (5) at (54bp,306bp) [vx0] {};
  \node (4) at (54bp,18bp) [vx0] {};
  \node (6) at (54bp,162bp) [vx1] {};
  \draw [->,ed1] (3) ..controls (36.748bp,208.01bp) and (40.462bp,198.1bp)  .. (6);
  \draw [->,ed0] (0) ..controls (44.252bp,424.01bp) and (40.538bp,414.1bp)  .. (2);
  \draw [->,ed0] (6) ..controls (63.987bp,136.32bp) and (67.894bp,126.27bp)  .. (1);
  \draw [->,ed0] (5) ..controls (59.747bp,334.38bp) and (61.997bp,347.88bp)  .. (63bp,360bp) .. controls (64.32bp,375.95bp) and (64.32bp,380.05bp)  .. (63bp,396bp) .. controls (62.294bp,404.52bp) and (60.973bp,413.73bp)  .. (0);
  \draw [->,ed0] (4) ..controls (48.891bp,46.402bp) and (46.891bp,59.902bp)  .. (46bp,72bp) .. controls (44.824bp,87.957bp) and (44.824bp,92.043bp)  .. (46bp,108bp) .. controls (46.627bp,116.51bp) and (47.802bp,125.71bp)  .. (6);
  \draw [->,ed1] (6) ..controls (59.747bp,190.38bp) and (61.997bp,203.88bp)  .. (63bp,216bp) .. controls (64.32bp,231.95bp) and (64.32bp,236.05bp)  .. (63bp,252bp) .. controls (62.294bp,260.52bp) and (60.973bp,269.73bp)  .. (5);
  \draw [->,ed0] (1) ..controls (72.013bp,64.319bp) and (68.106bp,54.274bp)  .. (4);
  \draw [->,ed1] (5) ..controls (44.252bp,280.01bp) and (40.538bp,270.1bp)  .. (3);
  \draw [->,ed0] (2) ..controls (36.748bp,352.01bp) and (40.462bp,342.1bp)  .. (5);
\end{pgfscope}%
\end{tikzpicture}} \\ \dd{3}{5} \\ $\mathsf{Aut}= 1$}

\\ \hline

\parbox[b]{\sizefour}{
\resizebox{\sizefour}{!}{\begin{tikzpicture}[>=latex,line join=bevel,]
\tikzstyle{vx0}=[circle, ball color=blue!80!black, inner sep=5bp];\tikzstyle{vx1}=[star, ball color=red, inner sep=4bp];\tikzstyle{ed0}=[>=triangle 45,ultra thick,dashed,red];\tikzstyle{ed1}=[>=triangle 45,thick,blue!80!black];\begin{pgfscope}[scale=0.5]%
\node (1) at (171bp,90bp) [vx0] {};
  \node (0) at (171bp,18bp) [vx0] {};
  \node (3) at (99bp,162bp) [vx0] {};
  \node (2) at (99bp,234bp) [vx0] {};
  \node (5) at (171bp,162bp) [vx0] {};
  \node (4) at (27bp,90bp) [vx0] {};
  \node (7) at (63bp,18bp) [vx1] {};
  \node (6) at (207bp,234bp) [vx0] {};
  \node (9) at (99bp,90bp) [vx1] {};
  \node (8) at (243bp,162bp) [vx1] {};
  \draw [->,ed0] (1) ..controls (171bp,64.131bp) and (171bp,54.974bp)  .. (0);
  \draw [->,ed0] (2) ..controls (99bp,208.13bp) and (99bp,198.97bp)  .. (3);
  \draw [->,ed0] (5) ..controls (145.8bp,136.8bp) and (132.68bp,123.68bp)  .. (9);
  \draw [->,ed0] (4) ..controls (39.96bp,64.081bp) and (45.154bp,53.693bp)  .. (7);
  \draw [->,ed1] (3) ..controls (99bp,136.13bp) and (99bp,126.97bp)  .. (9);
  \draw [->,ed1] (6) ..controls (194.04bp,208.08bp) and (188.85bp,197.69bp)  .. (5);
  \draw [->,ed1] (9) ..controls (86.04bp,64.081bp) and (80.846bp,53.693bp)  .. (7);
  \draw [->,ed1] (5) ..controls (171bp,136.13bp) and (171bp,126.97bp)  .. (1);
  \draw [->,ed0] (6) ..controls (219.96bp,208.08bp) and (225.15bp,197.69bp)  .. (8);
\end{pgfscope}%
\end{tikzpicture}} \dd{7}{4} \\ $\mathsf{Aut}= 1$}
&
\parbox[b]{\sizethree}{
\resizebox{\sizethree}{!}{\begin{tikzpicture}[>=latex,line join=bevel,]
\tikzstyle{vx0}=[circle, ball color=blue!80!black, inner sep=5bp];\tikzstyle{vx1}=[star, ball color=red, inner sep=4bp];\tikzstyle{ed0}=[>=triangle 45,ultra thick,dashed,red];\tikzstyle{ed1}=[>=triangle 45,thick,blue!80!black];\begin{pgfscope}[scale=0.5]%
\node (1) at (99bp,162bp) [vx0] {};
  \node (0) at (99bp,18bp) [vx0] {};
  \node (3) at (99bp,90bp) [vx0] {};
  \node (2) at (27bp,162bp) [vx0] {};
  \node (5) at (135bp,234bp) [vx0] {};
  \node (4) at (27bp,90bp) [vx0] {};
  \node (7) at (171bp,162bp) [vx1] {};
  \node (9) at (171bp,90bp) [vx1] {};
  \node (8) at (27bp,18bp) [vx1] {};
  \draw [->,ed1] (4) ..controls (27bp,64.131bp) and (27bp,54.974bp)  .. (8);
  \draw [->,ed0] (5) ..controls (147.96bp,208.08bp) and (153.15bp,197.69bp)  .. (7);
  \draw [->,ed0] (3) ..controls (73.803bp,64.803bp) and (60.685bp,51.685bp)  .. (8);
  \draw [->,ed1] (3) ..controls (99bp,64.131bp) and (99bp,54.974bp)  .. (0);
  \draw [->,ed1] (9) -- (3);
  \draw [->,ed1] (5) ..controls (122.04bp,208.08bp) and (116.85bp,197.69bp)  .. (1);
  \draw [->,ed1] (0) -- (9);
  \draw [->,ed0] (1) ..controls (99bp,136.13bp) and (99bp,126.97bp)  .. (3);
  \draw [->,ed0] (2) ..controls (27bp,136.13bp) and (27bp,126.97bp)  .. (4);
\end{pgfscope}%
\end{tikzpicture}} \dd{6}{4} \\ $\mathsf{Aut}= 1$}
&
\parbox[b]{\sizefour}{
\resizebox{\sizefour}{!}{\begin{tikzpicture}[>=latex,line join=bevel,]
\tikzstyle{vx0}=[circle, ball color=blue!80!black, inner sep=5bp];\tikzstyle{vx1}=[star, ball color=red, inner sep=4bp];\tikzstyle{ed0}=[>=triangle 45,ultra thick,dashed,red];\tikzstyle{ed1}=[>=triangle 45,thick,blue!80!black];\begin{pgfscope}[scale=0.5]%
\node (1) at (243bp,90bp) [vx0] {};
  \node (3) at (99bp,90bp) [vx0] {};
  \node (2) at (243bp,162bp) [vx0] {};
  \node (5) at (27bp,90bp) [vx0] {};
  \node (4) at (171bp,234bp) [vx0] {};
  \node (7) at (99bp,162bp) [vx1] {};
  \node (6) at (243bp,234bp) [vx0] {};
  \node (9) at (171bp,162bp) [vx1] {};
  \node (8) at (63bp,18bp) [vx1] {};
  \draw [->,ed1] (6) ..controls (243bp,208.13bp) and (243bp,198.97bp)  .. (2);
  \draw [->,ed0] (5) ..controls (39.96bp,64.081bp) and (45.154bp,53.693bp)  .. (8);
  \draw [->,ed0] (9) ..controls (145.8bp,136.8bp) and (132.68bp,123.68bp)  .. (3);
  \draw [->,ed1] (3) ..controls (86.04bp,64.081bp) and (80.846bp,53.693bp)  .. (8);
  \draw [->,ed0] (2) ..controls (243bp,136.13bp) and (243bp,126.97bp)  .. (1);
  \draw [->,ed1] (7) -- (4);
  \draw [->,ed1] (9) -- (7);
  \draw [->,ed0] (6) ..controls (217.8bp,208.8bp) and (204.68bp,195.68bp)  .. (9);
  \draw [->,ed1] (4) -- (9);
\end{pgfscope}%
\end{tikzpicture}} \dd{6}{4} \\ $\mathsf{Aut}= 1$}
& 
\parbox[b]{\sizefour}{
\resizebox{!}{\sizefour}{\begin{tikzpicture}[>=latex,line join=bevel,]
\tikzstyle{vx0}=[circle, ball color=blue!80!black, inner sep=5bp];\tikzstyle{vx1}=[star, ball color=red, inner sep=4bp];\tikzstyle{ed0}=[>=triangle 45,ultra thick,red,dashed];\tikzstyle{ed1}=[>=triangle 45,ultra thick,thick,blue!80!black];\begin{pgfscope}[scale=0.5]%
\node (1) at (108bp,162bp) [vx1] {};
  \node (0) at (108bp,306bp) [vx0] {};
  \node (3) at (27bp,234bp) [vx0] {};
  \node (2) at (54bp,378bp) [vx1] {};
  \node (5) at (54bp,306bp) [vx0] {};
  \node (4) at (108bp,234bp) [vx0] {};
  \node (6) at (54bp,162bp) [vx0] {};
  \node (7) at (-20bp,162bp) [vx1] {};
  \node (8) at (-20bp,306bp) [vx0] {};
  \draw [->,ed1] (7)  -- (8);
  \draw [->,ed1] (3)  -- (7);
  \draw [->,ed1] (8)  -- (3);
  \draw [->,ed0] (3) ..controls (36.748bp,208.01bp) and (40.462bp,198.1bp)  .. (6);
  \draw [->,ed1] (2) -- (0);
  \draw [->,ed1] (6) -- (1);
  \draw [->,ed1] (5) -- (0);
  \draw [->,ed1] (4) -- (6);
  \draw [->,ed0] (6) ..controls (59.747bp,190.38bp) and (61.997bp,203.88bp)  .. (63bp,216bp) .. controls (64.32bp,231.95bp) and (64.32bp,236.05bp)  .. (63bp,252bp) .. controls (62.294bp,260.52bp) and (60.973bp,269.73bp)  .. (5);
  \draw [->,ed1] (1) -- (4);
  \draw [->,ed0] (5) ..controls (44.252bp,280.01bp) and (40.538bp,270.1bp)  .. (3);
  \draw [->,ed1] (5) -- (2);
\end{pgfscope}%
\end{tikzpicture}} \\ \dd{4}{6} \\ $\mathsf{Aut}= 3$}
\\ \hline

\parbox[b]{\sizethree}{
\resizebox{\sizethree}{!}{\begin{tikzpicture}[>=latex,line join=bevel,]
\tikzstyle{vx0}=[circle, ball color=blue!80!black, inner sep=5bp];\tikzstyle{vx1}=[star, ball color=red, inner sep=4bp];\tikzstyle{ed0}=[>=triangle 45,ultra thick,dashed,red];\tikzstyle{ed1}=[>=triangle 45,thick,blue!80!black];\begin{pgfscope}[scale=0.5]%
  \node (0) at (99bp,234bp) [vx1] {};
\node (1) at (27bp,234bp) [vx0] {};
  \node (2) at (27bp,162bp) [vx0] {};
  \node (3) at (27bp,90bp) [vx0] {};
  \node (4) at (27bp,18bp) [vx0] {};
  \node (5) at (99bp,18bp) [vx0] {};
  \node (6) at (171bp,18bp) [vx1] {};
  \node (7) at (171bp,90bp) [vx0] {};
  \node (8) at (171bp,162bp) [vx0] {};
  \node (9) at (171bp,234bp) [vx1] {};
  \draw [->,ed0] (1) -- (0);
  \draw [->,ed1] (1) -- (2);
  \draw [->,ed0] (2) -- (3);
  \draw [->,ed1] (4) -- (3);
  \draw [->,ed1] (5) -- (4);
  \draw [->,ed1] (3) -- (5);
  \draw [->,ed0] (5) -- (6);
  \draw [->,ed0] (6) -- (7);
  \draw [->,ed0] (7) -- (5);
  \draw [->,ed1] (8) -- (7);
  \draw [->,ed0] (8) -- (9);
\end{pgfscope}%
\end{tikzpicture}} \dd{7}{4} \\ $\mathsf{Aut}= 1$}

&

\parbox[b]{\sizethree}{
\resizebox{\sizethree}{!}{\begin{tikzpicture}[>=latex,line join=bevel,]
\tikzstyle{vx0}=[circle, ball color=blue!80!black, inner sep=5bp];\tikzstyle{vx1}=[star, ball color=red, inner sep=4bp];\tikzstyle{ed0}=[>=triangle 45,ultra thick,dashed,red];\tikzstyle{ed1}=[>=triangle 45,thick,blue!80!black];\begin{pgfscope}[scale=0.5]%
  \node (0) at (99bp,234bp) [vx0] {};
\node (1) at (27bp,234bp) [vx0] {};
  \node (2) at (27bp,162bp) [vx1] {};
  \node (3) at (27bp,90bp) [vx0] {};
  \node (4) at (27bp,18bp) [vx1] {};
  \node (5) at (99bp,18bp) [vx0] {};
  \node (6) at (171bp,18bp) [vx0] {};
  \node (7) at (171bp,90bp) [vx1] {};
  \node (8) at (171bp,162bp) [vx0] {};
  \node (9) at (171bp,234bp) [vx0] {};
  \draw [->,ed0] (0) -- (1);
  \draw [->,ed1] (1) -- (2);
  \draw [->,ed0] (3) -- (2);
  \draw [->,ed1] (3) -- (5);
  \draw [->,ed1] (4) -- (3);
  \draw [->,ed1] (5) -- (4);
  \draw [->,ed0] (6) -- (5);
  \draw [->,ed0] (7) -- (6);
  \draw [->,ed0] (5) -- (7);
  \draw [->,ed1] (8) -- (7);
  \draw [->,ed0] (9) -- (8);
\end{pgfscope}%
\end{tikzpicture}} \dd{7}{4} \\ $\mathsf{Aut}= 1$}

&

\parbox[b]{\sizethree}{
\resizebox{\sizethree}{!}{\begin{tikzpicture}[>=latex,line join=bevel,]
\tikzstyle{vx0}=[circle, ball color=blue!80!black, inner sep=5bp];\tikzstyle{vx1}=[star, ball color=red, inner sep=4bp];\tikzstyle{ed0}=[>=triangle 45,ultra thick,dashed,red];\tikzstyle{ed1}=[>=triangle 45,thick,blue!80!black];\begin{pgfscope}[scale=0.5]%
  \node (0) at (99bp,234bp) [vx0] {};
\node (1) at (27bp,234bp) [vx1] {};
  \node (2) at (27bp,162bp) [vx0] {};
  \node (3) at (27bp,90bp) [vx0] {};
  \node (4) at (27bp,18bp) [vx0] {};
  \node (5) at (99bp,18bp) [vx1] {};
  \node (6) at (171bp,18bp) [vx0] {};
  \node (7) at (171bp,90bp) [vx0] {};
  \node (8) at (171bp,162bp) [vx1] {};
  \node (9) at (171bp,234bp) [vx0] {};
  \draw [->,ed0] (0) -- (1);
  \draw [->,ed1] (1) -- (2);
  \draw [->,ed0] (3) -- (2);
  \draw [->,ed1] (3) -- (4);
  \draw [->,ed1] (4) -- (5);
  \draw [->,ed1] (5) -- (3);
  \draw [->,ed0] (6) -- (5);
  \draw [->,ed0] (7) -- (6);
  \draw [->,ed0] (5) -- (7);
  \draw [->,ed1] (7) -- (8);
  \draw [->,ed0] (9) -- (8);
\end{pgfscope}%
\end{tikzpicture}} \dd{7}{4} \\ $\mathsf{Aut}= 1$}
&

\\ \hline
\end{longtable}

\end{document}